\documentclass[11pt,reqno]{amsart}
\usepackage{amssymb}
\usepackage{mathrsfs}
\usepackage{amsfonts}
\usepackage{amsfonts,amssymb,amsmath,amsthm}
\usepackage{url}
\usepackage{enumerate}
\usepackage[pdftex,bookmarksnumbered]{hyperref}
\usepackage{graphicx,xcolor}
\numberwithin{equation}{section}

\newtheorem{proposition}{Proposition}[section]
\newtheorem{theorem}{Theorem}[section]
\newtheorem{lemma}{Lemma}[section]

\newtheorem{definition}{Definition}[section]

\newtheorem{remark}{Remark}[section]
\newtheorem{example}{Example}[section]

\usepackage[left=2.0cm, right=2.0cm, top=2.0cm, bottom=2.0cm]{geometry}

\usepackage{graphicx}%
\usepackage{color}
\usepackage{cite}
\usepackage{hyperref}
\hypersetup{
	colorlinks = true,%
	citecolor = blue,
	filecolor=red,%
	linkcolor = [rgb]{0.65,0.0,0.0},%
	anchorcolor = red,
	pagecolor = red,
	urlcolor= [rgb]{0.65,0.0,0.0}
	linktocpage=true,
	pdfpagelabels=true,
	bookmarksnumbered=true,
}

\DeclareMathOperator{\vol}{vol}
\DeclareMathOperator{\di}{div}

\newcommand{\ds}{\displaystyle}

\author{Libing Huang}

\address{School of Mathematical Sciences and LPMC\\
 Nankai University\\
300071  Tianjin, China}
\email{huanglb@nankai.edu.cn}

\author{Alexandru Krist\'aly}
\address{Department of Economics\\
	Babe\c s-Bolyai University\\
	400591 Cluj-Napoca, Romania \&  Institute of Applied Mathematics\\
 \'Obuda University\\
 1034 Budapest, Hungary}
  \email{alex.kristaly@econ.ubbcluj.ro; kristaly.alexandru@nik.uni-obuda.hu}

\author{Wei Zhao}
\address{
Department of Mathematics\\
East China University of Science and Technology\\
200237 Shanghai, China}
\email{szhao\underline{ }wei@yahoo.com}
\thanks{
The research of A. Krist\'aly is supported by the National Research, Development and Innovation Fund of Hungary, financed under the K$\_$18 funding scheme, Project no.  127926. W. Zhao is supported by the National Natural Science Foundation of China (No. 11501202, No. 11761058) and the grant of China Scholarship Council (No. 201706745006).}

\begin{document}

\title[Sharp uncertainty principles on general Finsler manifolds]{Sharp uncertainty principles on general Finsler manifolds}

\begin{abstract}
The  paper is devoted to  sharp uncertainty principles (Heisenberg-Pauli-Weyl,  Caffarelli-Kohn-Nirenberg and Hardy inequalities) on forward complete Finsler manifolds endowed with an arbitrary measure. Under mild assumptions, the existence of extremals corresponding to the sharp constants in the Heisenberg-Pauli-Weyl and  Caffarelli-Kohn-Nirenberg  inequalities fully \textit{characterizes} the nature of the Finsler manifold in terms of three non-Riemannian quantities, namely, its \textit{reversibility} and the vanishing of the \textit{flag curvature} and $S$\textit{-curvature} induced by the measure, res\-pectively. It turns out in particular that the Busemann-Hausdorff measure is the optimal one in the study of sharp uncertainty principles on Finsler manifolds.
The  optimality of our results are supported by Randers-type Finslerian examples originating from the Zermelo navigation problem.

%
%


\end{abstract}
\maketitle

\section{Introduction} \label{sect1}

Given $p,q\in \mathbb{R}$ and $n\in \mathbb{N}$ with $0<q<2<p$ and $2<n<\frac{2(p-q)}{p-2}$,
the Caffarelli-Kohn-Nirenberg interpolation inequality in the Euclidean space $\mathbb{R}^n$ states that
\begin{equation}\label{new1.1}
\left(\ds\int_{\mathbb{R}^n} |\nabla u(x)|^2dx \right)\left( \ds\int_{\mathbb{R}^n}\frac{|u(x)|^{2p-2}}{|x|^{2q-2}}dx \right)\geq \frac{(n-q)^2}{p^2}\left(\ds\int_{\mathbb{R}^n}\frac{|u(x)|^p}{|x|^q}dx\right)^2,\ \ \forall u\in C^\infty_0(\mathbb{R}^n),
\end{equation}
where the constant $\frac{(n-q)^2}{p^2}$ is sharp and the corresponding extremal functions are  $u(x)=(C+|x|^{2-q})^{\frac1{2-p}}$, $C>0$ (up to scalar multiplication and translation).

When $p\rightarrow 2$ and $q\rightarrow 0$, inequality (\ref{new1.1}) turns to be the Heisenberg-Pauli-Weyl principle, i.e.,
\[
\left(\ds\int_{\mathbb{R}^n} |\nabla u(x)|^2dx \right)\left( \ds\int_{\mathbb{R}^n}|x|^2u^2(x)dx  \right)\geq \frac{n^2}{4}\left(\ds\int_{\mathbb{R}^n} u^2(x)dx  \right)^2,\ \ \forall u\in C^\infty_0(\mathbb{R}^n).\tag{1.2}\label{new1.2}
\]
Here, the constant $\frac{n^2}4$ is sharp while the extremal functions become  the Gaussian functions $u(x)=e^{-C|x|^2}$, $C>0$ (up to scalar multiplication and translation). When  $p\rightarrow2$ and $q\rightarrow 2$, (\ref{new1.1}) reduces to the Hardy inequality, i.e.,
\[
\ds\int_{\mathbb{R}^n}|\nabla u(x)|^2dx \geq \frac{(n-2)^2}{4}\ds\int_{\mathbb{R}^n} \frac{u^2(x)}{|x|^2}dx,\ \ \forall u\in C^\infty_0(\mathbb{R}^n). \tag{1.3}\label{new1.3}
\]
In this case, the constant $\frac{(n-2)^2}{4}$ is still sharp but there are no extremal functions.

Inequality  (\ref{new1.2}) is viewed as the sharp PDE form of the well-known uncertainty principle in quantum mechanics which states that the position and momentum of a given particle cannot be accurately determined simultaneously.  Inequality (\ref{new1.3}) has been proved first by Leray \cite{Leray} for $n=3$ by studying the Navier-Stokes equations, which is the extension of the one-dimensional  inequality by  Hardy \cite{Hardy}. Accordingly, some authors attribute to  (\ref{new1.3}) the name of Hardy-Leray inequality; see Opic and Kufner \cite{OK} for a historical presentation. On the other hand, (\ref{new1.3}) can be viewed as well as an uncertainty principle, see e.g. Brezis and V\'azquez \cite[p. 452]{BV}, Frank \cite{Frank} and Lieb \cite{Lieb}, since if $u$ is localized close to $x= 0$ (i.e., the right hand side of (\ref{new1.3}) is large due to the presence of the singularity term  $|x|^{-2}$), then its momentum is large as well. Thus, for the above three inequalities (\ref{new1.1})-(\ref{new1.3}) we shall adopt the common notion of {\it uncertainty principles}.  We also notice that a simply combination of the Hardy inequality (\ref{new1.3}) and H\"older inequality provides a non-sharp form of the Heisenberg-Pauli-Weyl principle (\ref{new1.2}).
Important contributions in the theory of uncertainty principles can be found in
  Adimurthi, Chaudhuri and Ramaswamy \cite{ACR}, Barbatis,  Filippas and Tertikas  \cite{BFT},  Caffarelli,  Kohn and Nirenberg \cite{CKN}, Erb \cite{Er}, Fefferman \cite{F}, Filippas and Tertikas \cite{FT}, Ghoussoub and Moradifam \cite{GM,GM2}, Ruzhansky and Suragan \cite{RS1, RS2, RS3},  Wang and Willem \cite{WW} and subsequent references.


Certain uncertainty principles   have also been investigated in \textit{curved spaces}. As far as we know, Carron \cite{Ca} was the first who studied weighted $L^2$-Hardy inequalities  on complete, non-compact Riemannian manifolds. On one hand, inspired by \cite{Ca}, a systematic study of the Hardy inequality is carried out by Berchio,  Ganguly and  Grillo \cite{BGD}, D'Ambrosio and Dipierro \cite{DD},   Kombe and \"Ozaydin \cite{KO,KO2}, Yang, Su and Kong \cite{YSK} in the Riemannian setting, as well as by Krist\'aly and Repov\v s  \cite{KR} and Yuan, Zhao and Shen \cite{YZY} in the Finsler setting. On the other hand, Caffarelli-Kohn-Nirenberg-type inequalities are studied by do Carmo and Xia \cite{CX}, Erb \cite{Er}
and Xia \cite{X} on  Riemannian manifolds, and by Krist\'aly \cite{Kristaly-JGA} and Krist\'aly and Ohta \cite{KOh} on  Finsler manifolds.

Very recently, Krist\'aly \cite{K} fully described the influence of
\textit{curvature} to  uncertainty principles in the Riemannian setting; these results can be  summarized as follows:
\medskip

\noindent\textbf{Statement 1.} (Non-positively curved case)  All three uncertainty principles hold on Riemannian Cartan-Hadamard manifolds (simply connected, complete Riemannian manifolds with non-positive sectional curvature) with the
same sharp constants as in their Euclidean counterparts. Moreover, the existence of positive extremals
corresponding to  the sharp constants in the Heisenberg-Pauli-Weyl and  Caffarelli-Kohn-Nirenberg interpolation inequalities implies the flatness of the Riemannian manifold.

\medskip

\noindent\textbf{Statement 2.} (Non-negatively curved case) When a complete Riemannian manifold has non-negative Ricci curvature, the validity of Heisenberg-Pauli-Weyl or Caffarelli-Kohn-Nirenberg
interpolation inequality with its sharp Euclidean constant implies the flatness of  the Riemannian
manifold.

\medskip

Although the second author pointed out in the unpublished paper \cite{K2} that Statements 1 and 2 can be extended to \textit{reversible Berwald spaces} (Finsler manifolds whose tangent spaces  are linearly isometric to a common Minkowski space) equipped with the Busemann-Hausdorff measure, the purpose of the present paper is to investigate uncertainty principles on generic Finsler manifolds $(M,F)$ endowed with an arbitrary measure $d\mathfrak{m}$. In such a setting, the Euclidean quantities $|\nabla u(x)|$, $|x|$ and $dx$ from (\ref{new1.1})-(\ref{new1.3}) are naturally replaced by the co-Finslerian norm of the differential $F^*(du)$ (or max$\{F^*(\pm du)\}$, or min$\{F^*(\pm du)\}$), the Finsler distance function $d_F$, and the measure $d\mathfrak{m}$, respectively.  In spite of the fact that  Chern \cite{Chern} claimed that 'Finsler geometry is just Riemannian geometry without the
quadratic restriction', subtle differences occur between these geometries.

 In order to emphasize the contrast between the Riemannian and Finslerian settings within the theory of uncertainty principles, we start with two simple examples that will be detailed in the Appendix.  First,   for  $t\in [0,1)$ we consider on $\mathbb R^2$ the perturbation of the Euclidean metric as
\[
\label{elso-minko-norma}\tag{1.4}
F_t(x,y)=|y|+ty^2,\ \ y=(y^1,y^2)\in \mathbb R^2.
\]
The pair $(\mathbb R^2,F_t)$ is a \textit{Minkowski space} (of Randers type), thus having vanishing flag and $S$-curvatures, respectively. It turns out that the Finslerian  Heisenberg-Pauli-Weyl principle holds on $(M,F_t)$ for every $t\in [0,1)$ with the sharp constant $\frac{n^2}{4}=1$, but extremal functions exist if and only if $t=0$, i.e., $F_t=F_0$ is reversible (in particular, $F_0$ is Euclidean), see Example \ref{simexa}. Second, we observe that on the $n$-dimensional  Euclidean open unit ball $B^n$  $(n\geq 3)$ endowed with the \textit{Funk metric} $F$ (see Shen \cite{Sh1}),
 the Hardy inequality \textit{fails}, see Example \ref{example-2}. More precisely, in spite of the fact that $(B^n,F)$ is simply connected, forward complete and has constant flag curvature $-\frac{1}{4}$ (thus Statement 1 formally applies), it turns out that
\[
\label{Funk-nulla}\tag{1.5}
\inf_{u\in C^\infty_0(B^n)\backslash\{0\}}\frac{\ds \int_{B^n}F^{*2}{(du)}d\mathfrak{m}_{BH} }{\ds \int_{B^n}\frac{u^2}{\rho_{\textbf{0}}^2}d\mathfrak{m}_{BH}}=0,
\]
where $d\mathfrak{m}_{BH}$ is the Busemann-Hausdorff measure on $(B^n,F)$,    $\textbf{0}=(0,...,0)\in \mathbb R^n$ and $\rho_{\textbf{0}}(\cdot)=d_F(\textbf{0},\cdot)$.
We notice that $(B^n,F)$  has \textit{infinite} reversibility and \textit{non-vanishing} $S$-curvature.

A closer inspection of the above instructive examples shows that while on Riemannian manifolds only the sectional {curvature} has a deciding role (cf. Statements 1\&2), on Finsler manifolds three \textit{non-Riemannian quantities} will influence the validity and the existence of extremal functions in the uncertainty principles, as
\begin{enumerate}
	\item[$\bullet$] reversibility;
	\item[$\bullet$] $S$-curvature induced by the given measure;
	\item[$\bullet$] flag curvature.
\end{enumerate}
Clearly, in the Riemannian setting the first two quantities naturally disappear, while the flag curvature coincides with the usual sectional curvature.

In order to state our main results, we briefly recall the aforementioned three notions (for details, see Section \ref{Sec2}).
Throughout the paper,  $(M,F)$ is an $n$-dimensional non-compact  Finsler manifold. The  {\it reversibility} of $(M,F)$, introduced by Rademacher \cite{R}, is given by
\[
\lambda_F(M):=\sup_{x\in M} \lambda_F(x),\ \ {\rm where}\ \ \lambda_F(x)= \sup_{y\in T_xM\setminus\{0\}} \frac{F(x,-y)}{F(x,y)}.
\]
It is easy to see that $\lambda_F(M)\geq 1$ with equality if and only if $F$ is reversible (i.e., symmetric). Clearly, Riemannian metrics are always reversible. However, there are infinitely many non-reversible Finsler metrics; for example, a Randers metric $F=\alpha+\beta$ is reversible on a manifold $M$ (where $\alpha$ is
a Riemannian metric on $M$ and $\beta$ is a $1$-form with $ \|\beta\|_\alpha:=\sqrt{\alpha(\beta,\beta)}<1$) if and only if $\beta=0.$

Unlike in the Riemannian setting (where the canonical Riemannian measure is used), on a Finsler manifold  various measures can be introduced whose behavior may be genuinely different. Two such frequently used measures are the so-called  Busemann-Hausdorff measure $d\mathfrak{m}_{BH}$ and  Holmes-Thompson measure $d\mathfrak{m}_{HT}$, see Alvarez-Paiva and  Berck\cite{AlB} and Alvarez-Paiva and Thompson \cite{AlT}. In particular, these measures for a Randers metric $F=\alpha+\beta$ are
\[
d\mathfrak{m}_{BH}=\left(1-\|\beta\|_\alpha^2\right)^\frac{n+1}2 dV_\alpha,\ d\mathfrak{m}_{HT}=dV_\alpha,
\]
where $dV_\alpha$ is the Riemannian measure induced by the Riemannian metric $\alpha$. The densities of these measures show that $d\mathfrak{m}_{BH}\leq d \mathfrak{m}_{HT}$ with equality if and only if $F$ is Riemannian (i.e., $\beta=0$).

%
%
%
%
%




   An arbitrary measure $d\mathfrak{m}$ on a Finsler manifold $(M,F)$ induces
   two further  non-Riemannian quantities $\tau$ and $\mathbf{S}$,  see Shen \cite{Sh1}, which are the so-called  {\it distortion}  and {\it S-curvature}, respectively. More precisely, if $d\mathfrak{m}:=\sigma(x)dx^1\wedge...\wedge dx^n$ in some local coordinate $(x^i)$, for any $y\in T_xM\backslash\{0\}$, let
\begin{equation*}
\tau(y):=\log \frac{\sqrt{\det g_{ij}(x,y)}}{\sigma(x)},\ \ \ \  \mathbf{S}(y):=\left.\frac{d}{dt}\right|_{t=0}[\tau(\dot{\gamma}_y(t))],
\end{equation*}
where $g_y=(g_{ij}(x,y))$ is the fundamental tensor induced by $F$ and $t\mapsto \gamma_y(t)$ is the geodesic starting at $x\in M$ with $\dot{\gamma}_y(0)=y\in T_xM$. In particular, the $S$-curvature $\textbf{S}_{BH}$  of the measure $d\mathfrak{m}_{BH}$  vanishes on any Berwald space (including both Riemannian manifolds and Minkowski spaces), see Shen \cite{Shen_Adv_Math,Shen2013}.

The measures $d\mathfrak{m}_1$ and $d\mathfrak{m}_2$ are  {\it equivalent} if there exists  $C>0$ such that $d\mathfrak{m}_1=C d\mathfrak{m}_2$; the  {\it equivalence} {\it class} of $d\mathfrak{m}$ is denoted by $[d\mathfrak{m}]$. Clearly, the $S$-curvatures of two equivalent measures coincide.

Let
\[
\mathscr{L}_{\mathfrak{m}}(x):=\frac{1}{n}\ds\int_{S_xM}e^{-\tau(y)}d\nu_x(y),
\]
where  $S_xM:=\{y\in T_xM:\, F(x,y)=1\}$ is the indicatrix at $x$ and $d\nu_x$ is the Riemannian measure on $S_xM$ induced by $F$.

Let $P:=\text{Span}\{y,v\}\subset T_xM$ be a plane. The \textit{flag curvature} is defined by
\[
\mathbf{K}(y,v):=\frac{g_y\left( R_y(v),v  \right)}{g_y(y,y)g_y(v,v)-g^2_y(y,v)},
\]
where $R_y$ is the Riemannian curvature  of $F$. By means of the flag curvature, one can define in the usual way the  {\it Ricci curvature} ${\bf Ric}$. A Finsler manifold $(M,F)$ is \textit{Cartan-Hadamard} if it is forward complete, simply connected with $\mathbf{K}\leq 0.$

In the sequel we  suppose that $p,q\in \mathbb{R}$ and $n\in \mathbb{N}$ satisfy one of the following conditions:
 \[
\left\{
\begin{array}{lll}
\text{(I)  }& p=2, q=0 \text{ and } n\geq 2;\\
\tag{1.6}\label{1.1}\\
\text{(II)  }&0<q<2<p \text{ and }2<n<\frac{2(p-q)}{p-2}.\\
\end{array}
\right.
\]
Set $\rho_{x}(\cdot):=d_F(x,\cdot)$ and
\[
J^{\rm max}_{p,q}(x,u):=\frac{\left(\ds\ds\int_M \max \{F^{*2}(\pm du)\}  d\mathfrak{m} \right)\left( \ds\ds\int_M \frac{|u|^{2p-2}}{\rho^{2q-2}_{x}}    d\mathfrak{m} \right)}{\left(\ds\ds\int_M  \frac{|u|^p}{\rho^q_{x}} d\mathfrak{m}\right)^2}, \ x\in M,\ u\in C^\infty_0(M)\setminus\{0\}.\tag{1.7}\label{1.2}
\]

 Our first main result reads as follows.

\begin{theorem}\label{fistCan} Let $p,q\in \mathbb{R}$ and $n\in \mathbb{N}$ satisfying one of the  conditions of $(\ref{1.1})$ and let $(M,F, d\mathfrak{m})$ be an $n$-dimensional Cartan-Hadamard  manifold with
 $ \mathbf{S}\leq 0$.
Then we have:

\begin{itemize}
	\item[{\rm (i)}] For every $x\in M$, \[
J^{\rm max}_{p,q}(x,u)\geq \frac{(n-q)^2}{p^2},\ \ \forall u\in C^\infty_0(M)\setminus\{0\}.\eqno{({\rm {\mathbf J}}_{p,q,x}^{\rm max})}
\]  Moreover, if $\lambda_F(M)=1$, then $\frac{(n-q)^2}{p^2}$ is sharp, i.e., for every $x\in M$,
	\[
	\inf_{u\in C^\infty_0(M)\setminus\{0\}}J^{\rm max}_{p,q}(x,u)=\frac{(n-q)^2}{p^2}.
	\]
	\item[{\rm (ii)}] Assume that  $d\mathfrak{m}=d\mathfrak{m}_{BH}$ and there exists a point $x_0\in M$ such that
$
	\lambda_F(x_0)=\lambda_F(M).
$
	Then the following statements are equivalent:
	
	\smallskip
	
	{\rm (a)} $\frac{(n-q)^2}{p^2}$ is achieved by an extremal in $({\rm {\mathbf J}}_{p,q,x_0}^{\rm max});$
	
	\smallskip
	
	{\rm (b)} $\frac{(n-q)^2}{p^2}$ is achieved by an extremal  in $({\rm {\mathbf J}}_{p,q,x}^{\rm max})$ for every $x\in M;$
	
	\medskip
	
	{\rm (c)} $(M,F)$ satisfies $\lambda_F(M)=1$, $\mathbf{K}=0$ and $\mathbf{S}_{BH}=0$.
	
	\item[{\rm (iii)}] Assume that  $\lambda_F(M)=1$ and  there exists a point $x_0\in M$ such that
	$
	\mathscr{L}_\mathfrak{m}(x_0)= \inf_{x\in M}\mathscr{L}_\mathfrak{m}(x).
	$
	Then the following statements are equivalent:
	
	\smallskip
	
	{\rm (a)} $\frac{(n-q)^2}{p^2}$ is achieved by an extremal in $({\rm {\mathbf J}}_{p,q,x_0}^{\rm max});$
	
	\smallskip
	
	{\rm (b)} $\frac{(n-q)^2}{p^2}$ is achieved by an extremal in $({\rm {\mathbf J}}_{p,q,x}^{\rm max})$ for every $x\in M;$
	
	\medskip
	
	{\rm (c)} $(M,F,d\mathfrak{m})$ satisfies $d\mathfrak{m}\in [d\mathfrak{m}_{BH}]$,
	$
	\mathbf{K}=0$ and $ \mathbf{S}=\mathbf{S}_{BH}=0.
	$
\end{itemize}
\end{theorem}
It is easy to see that   (I) and (II) in (\ref{1.1}) correspond to the Heisenberg-Pauli-Weyl principle and Caffarelli-Kohn-Nirenberg interpolation inequality, respectively. In particular, Theorem \ref{fistCan} implies Statement 1 in Krist\'aly \cite{K,K2}.

By considering $\ds\ds\int_M  F^{*2}(du)  d\mathfrak{m}$ instead of $\ds\ds\int_M  \max\{F^{*2}(\pm du)\}  d\mathfrak{m}$, we obtain a slightly different version of Theorem \ref{fistCan}. Set
\[
J_{p,q}(x,u):=\frac{\left(\ds\int_M  F^{*2}(du)  d\mathfrak{m} \right)\left( \ds\int_M \frac{|u|^{2p-2}}{\rho^{2q-2}_{x}}    d\mathfrak{m} \right)}{\left(\ds\int_M  \frac{|u|^p}{\rho^q_{x}} d\mathfrak{m}\right)^2}, \ x\in M,\ u\in C^\infty_0(M)\setminus\{0\}.\tag{1.8}\label{1.3}
\]

\begin{theorem}\label{thirdCan} Let $p,q\in \mathbb{R}$ and $n\in \mathbb{N}$ satisfying one of the  conditions of $(\ref{1.1})$  and let $(M,F, d\mathfrak{m})$ be an $n$-dimensional Cartan-Hadamard manifold with
$\mathbf{S}\leq 0$  and $ \lambda_F(M)<+\infty.
$ Then for every $x\in M,$
$$
J_{p,q}(x,u)\geq \frac{(n-q)^2}{p^2\lambda_F^2(M)},\ \ \forall u\in C^\infty_0(M)\setminus\{0\}. \eqno{({\rm {\mathbf J}}_{p,q,x})}$$
  Moreover,  assume that there exists a point $x_0\in M$ such that
$
\mathscr{L}_\mathfrak{m}(x_0)= \inf_{x\in M}\mathscr{L}_\mathfrak{m}(x).
$
Then the following statements are equivalent:

\begin{itemize}
	\item[{\rm (a)}] $\frac{(n-q)^2}{p^2\lambda_F^2(M)}$ is achieved by an extremal  in $({\rm {\mathbf J}}_{p,q,x_0});$
	
	\item[{\rm (b)}] $\frac{(n-q)^2}{p^2\lambda_F^2(M)}$ is achieved by an extremal  in $({\rm {\mathbf J}}_{p,q,x})$ for every $x\in M;$
	
	\item[{\rm (c)}] $(M,F,d\mathfrak{m})$  satisfies $d\mathfrak{m}\in [d\mathfrak{m}_{BH}]$,
$
	\lambda_F(M)=1,$ $ \mathbf{K}=0$ and $ \mathbf{S}=\mathbf{S}_{BH}=0.$
	
\end{itemize}
\end{theorem}

\medskip
\noindent Clearly, Theorem \ref{thirdCan} coincides with Theorem \ref{fistCan}/(iii) in the reversible case. If the sharp constants in Theorems \ref{fistCan} \& \ref{thirdCan} are achieved at some point $x_0$,
the extremals   (up to a positive scalar multiplication) are
 \[
u(x)=\left\{
\begin{array}{lll}
e^{-C\rho^2_{x_0}(x)}& \text{ if } p=2, q=0 \text{ and } n\geq 2,\\
(C+\rho_{x_0}(x)^{2-q})^{\frac1{2-p}}&\text{ if }0<q<2<p \text{ and }2<n<\frac{2(p-q)}{p-2},\\
\end{array}
\right.\ \ \mbox{where } C>0.
\]

\medskip

In the sequel we are going to study Finsler manifolds with non-negative Ricci curvature, obtaining  an extension of Statement 2 from Krist\'aly \cite{K,K2} to Finsler manifolds. To do this, set
\[
J^{\rm min}_{p,q}(x,u):=\frac{\left(\ds\int_M \min \{F^{*2}(\pm du)\}  d\mathfrak{m} \right)\left( \ds\int_M \frac{|u|^{2p-2}}{\rho^{2q-2}_{x}}    d\mathfrak{m} \right)}{\left(\ds\int_M  \frac{|u|^p}{\rho^q_{x}} d\mathfrak{m}\right)^2},\ x\in M, \ u\in C^\infty_0(M)\setminus\{0\}.\tag{1.9}\label{1.4}
\]
\begin{theorem}\label{forCan}
 Let $p,q\in \mathbb{R}$ and $n\in \mathbb{N}$ satisfying one of the  conditions of $(\ref{1.1})$  and let $(M,F, d\mathfrak{m})$ be an $n$-dimensional  forward  complete Finsler manifold with
$
\mathbf{Ric}\geq 0$ and $ \mathbf{S}\geq 0.
$

\begin{itemize}
	\item[{\rm (i)}] Assume that  $d\mathfrak{m}=d\mathfrak{m}_{BH}$ and there exists a point $x_0\in M$ such that
	$
	\lambda_F(x_0)=\lambda_F(M).
	$
	Then the following statements are equivalent:
	
	\begin{itemize}
		\item[{\rm (a)}] $J^{\rm min}_{p,q}(x_0,u)\geq \frac{(n-q)^2}{p^2}$ for every $u\in C^\infty_0(M)\setminus\{0\};$
		\item[{\rm (b)}]  $J^{\rm min}_{p,q}(x,u)\geq \frac{(n-q)^2}{p^2}$ for every $u\in C^\infty_0(M)\setminus\{0\}$ and  $x\in M;$
		\item[{\rm (c)}] $(M,F)$ satisfies $\lambda_F(M)=1$, $\mathbf{K}=0$ and $\mathbf{S}_{BH}=0$.
	\end{itemize}
	
	\item[{\rm (ii)}]  Assume that  $\lambda_F(M)=1$ and  there exists a point $x_0\in M$ such that
$
	\mathscr{L}_\mathfrak{m}(x_0)= \sup_{x\in M}\mathscr{L}_\mathfrak{m}(x).
$
	Then the following statements are equivalent:

		\begin{itemize}
		\item[{\rm (a)}] $J^{\rm min}_{p,q}(x_0,u)\geq \frac{(n-q)^2}{p^2}$ for every $u\in C^\infty_0(M)\setminus\{0\};$
		\item[{\rm (b)}] $J^{\rm min}_{p,q}(x,u)\geq \frac{(n-q)^2}{p^2}$ for every $u\in C^\infty_0(M)\setminus\{0\}$ and  $x\in M;$
		\item[{\rm (c)}] $(M,F,d\mathfrak{m})$ satisfies $d\mathfrak{m}\in [d\mathfrak{m}_{BH}]$,
		$
		\mathbf{K}=0 \text{ and } \mathbf{S}=\mathbf{S}_{BH}=0.
		$
	\end{itemize}
\end{itemize}
\end{theorem}
\medskip
\begin{remark} \rm
The conditions $\mathbf{Ric}\geq 0$, $\mathbf{S}\geq 0$ in Theorem \ref{forCan} give an upper bound of $\mathfrak{m}(B^+_{x_0}(r))$, which is indispensable in our proof. We note that another important Ricci curvature in Finsler geometry is the weighted Ricci curvature $\mathbf{Ric}_N$ for $N\in [n,\infty)$, see Ohta and Sturm \cite{Ot}. However, this curvature is more suitable to study the relative volume comparison rather than estimate the volume of small balls, cf. Ohta \cite{O}. Moreover, if $\mathbf{Ric}_N\geq 0$ and there exist two positive constants $C,\epsilon$ such that $\mathfrak{m}(B^+_{x_0}(r))\leq C r^N$ for $r\in (0,\epsilon)$, then Ohta \cite[Theorem 1.2]{O}  together with Zhao and Shen \cite[Lemma 3.1]{ZS} furnishes $\mathbf{Ric}=\mathbf{Ric}_N\geq 0$, $\mathbf{S}=0$ and $N=n$.
\end{remark}

\begin{remark}\label{elso-remark} \rm (i) On one hand,  Theorems \ref{fistCan}-\ref{forCan} show that the Busemann-Hausdorff measure is the 'optimal' one to study sharp uncertainty principles on  Finsler manifolds. In particular, if we apply Theorems \ref{fistCan}-\ref{forCan} on a \textit{reversible} Berwald space $(M,F)$ equipped with the Holmes-Thompson measure $d\mathfrak{m}_{HT}$, it turns out from our proof that $\mathscr{L}_{\mathfrak{m}_{HT}}$ is a constant and the $S$-curvature induced by  $d\mathfrak{m}_{HT}$ vanishes; therefore,
$d\mathfrak{m}_{HT}=C d\mathfrak{m}_{BH}$ for some $0<C\leq 1$ with equality if and only if $F$ is Riemannian.
	On the other hand, Theorems \ref{fistCan}-\ref{forCan} also show that even on simplest \textit{non-reversible} Berwald spaces (equipped with the Busemann-Hausdorff measure) the sharp  constants \textit{cannot} be achieved in sharp uncertainty principles; the Minkowski space $(\mathbb R^2,F_t)$  in (\ref{elso-minko-norma}) falls precisely into this class whenever $t>0$.

	(ii)  According to Theorems \ref{fistCan}-\ref{forCan}, the existence of extremals corresponding to the sharp
	constants implies the vanishing of both the flag curvature and  $S$-curvature induced by $d\mathfrak{m}_{BH}$. 
	A well-known fact   is that a flat Riemannian manifold $(M^n,g)$ is always locally isometric to $\mathbb{R}^n$ and  is  globally isometric to $\mathbb{R}^n$ whenever  $(M^n,g)$ is simply-connected and complete. Intuitively, a Finsler manifold with $\mathbf{K}=0$ and $\mathbf{S}_{BH}=0$ should be (at least locally) Minkowskian. However, this is \textit{not} true in general, see Shen \cite{Shen2003}. In fact, by using the Zermelo navigation problem we construct in the Appendix a whole class of examples   which satisfy these curvature vanishing properties but are not Berwaldian (hence, not Minkowskian); all these examples are \textit{non-complete} Finsler manifolds. However, if we suppose additionally that the Finsler manifold is either \textit{reversible} or \textit{forward complete}, all such examples are \textit{Minkowskian}, see e.g. Shen \cite[Theorem 1.2]{Shen2003} for Randers spaces. Up to now,  no full classification is available concerning this issue.

\end{remark}

%
%
We conclude this section by  considering the Hardy inequality, i.e., $p=q=2$ and $n\geq 3$. 
\begin{theorem}\label{Hardyineq}Given $n\geq 3$,
let $(M,F, d\mathfrak{m})$ be an $n$-dimensional forward complete Finsler manifold with $\mathbf{K}\leq 0$ and $\mathbf{S}\leq 0$. Then $$J^{\rm max}_{2,2}(x,u)\geq \frac{(n-2)^2}{4},\ \ \forall x\in M,\ u\in C^\infty_0(M)\setminus\{0\}.$$
In addition, if $F$ is reversible, then the constant $\frac{(n-2)^2}{4}$ is sharp but never achieved.
\end{theorem}

We note that Theorems \ref{fistCan},  \ref{thirdCan} and
\ref{Hardyineq} (resp., Theorem \ref{forCan}) can be established under the assumption $\mathbf{K}\leq 0,$ $\mathbf{S}\geq 0$ (resp., $\mathbf{Ric}\geq 0,$ $\mathbf{S}\leq 0$) and for backward complete Finsler manifolds; we leave the formulation of such statements to the interested reader.

\medskip

The paper is organized as follows. Section \ref{Sec2} is devoted to preliminaries on Finsler geometry together with some fine properties of the integral of distortion. In Section \ref{Sec4}  the Heisenberg-Pauli-Weyl principle, in Section \ref{Sec5}  the Caffarelli-Kohn-Nirenberg interpolation inequality, while in Section  \ref{section-Hardy} the Hardy inequality is  discussed.  The Appendix is devoted to the detailed discussion of the examples mentioned in (\ref{elso-minko-norma}) and (\ref{Funk-nulla}) as well as the  construction of some non-Berwaldian spaces with $\mathbf{K}=0$ and $\mathbf{S}_{BH}=0$, respectively, inspired by the Zermelo navigation problem.

\section{Preliminaries}\label{Sec2}
\subsection{Elements from Finsler geometry} In this section, we recall some definitions and properties from Finsler geometry; for details see  Bao, Chern and Shen \cite{BCS} and Shen \cite{Shen2013,Sh1}.

\subsubsection{Finsler manifolds.}
%
%
%

 Let $M$ be a connected
$n$-dimensional smooth manifold and $TM=\bigcup_{x \in M}T_{x}
M $ be its tangent bundle. The pair $(M,F)$ is a \textit{Finsler
	manifold} if the continuous function $F:TM\to [0,\infty)$ satisfies
the conditions

(a) $F\in C^{\infty}(TM\setminus\{ 0 \});$

(b) $F(x,\lambda y)=\lambda F(x,y)$ for all $\lambda\geq 0$ and $(x,y)\in TM;$

(c) $g_y:=g_{ij}(x,y)=[\frac12F^{2}%
]_{y^{i}y^{j}}(x,y)$ is positive definite for all $(x,y)\in
TM\setminus\{ 0 \}$ where $F(x,y)=F(y^i\frac{\partial}{\partial x^i}|_x)$.

Let $\pi:PM\rightarrow M$ and $\pi^*TM$ be the projective sphere bundle and the pullback bundle, respectively. The
Finsler metric $F$ induces a natural Riemannian metric $g=g_{ij}(x,[y])\,d\mathfrak{x}^i\otimes d\mathfrak{x}^j$, which is  the so-called {\it fundamental tensor} on  $\pi^*TM$, where
\[
g_{ij}(x,[y]):=\frac12\frac{\partial^2
F^2(x,y)}{\partial y^i\partial
y^j}, \ d\mathfrak{x}^i=\pi^*d x^i.
\]
The Euler theorem yields that $F^2(x,y)=g_{ij}(x,[y])y^iy^j$ for every $(x,y)\in TM\backslash\{0\}$. Note that $g_{ij}$ can be viewed as a local function on $TM\backslash\{0\}$, but it cannot be defined at $y=0$ unless $F$ is Riemannian.

The dual Finsler metric $F^*$ of $F$ on $M$ is
defined by
\begin{equation*}
F^*(x,\eta):=\underset{y\in T_xM\backslash \{0\}}{\sup}\frac{\eta(y)}{F(x,y)}, \ \
\forall \eta\in T_x^*M,
\end{equation*}
which is also a Finsler metric on $T^*M$.
The Legendre transformation $\mathfrak{L} : TM \rightarrow T^*M$ is defined
by
\begin{equation*}
\mathfrak{L}(X):=\left \{
\begin{array}{lll}
& g_X(X,\cdot) & \ \ \ X\neq0, \\
& 0 & \ \ \ X=0.%
\end{array}
\right.
\end{equation*}
In particular,  $\mathfrak{L}:TM\backslash\{0\}\rightarrow T^*M\backslash\{0\}$ is a diffeomorphism with $F^*(\mathfrak{L}(X))=F(X)$, $X\in TM$.
Now let $f : M \rightarrow \mathbb{R}$ be a smooth function on $M$; the
gradient of $f$ is defined as $\nabla f = \mathfrak{L}^{-1}(df)$. Thus,  $df(X) = g_{\nabla f} (\nabla f,X)$.

Let $\varphi$ be a  piecewise
$C^1$-function on $M$ such that every $\varphi^{-1}(t)$ is compact. The (area) measure on $\varphi^{-1}(t)$ is defined by
$dA:=(\nabla \varphi)\rfloor d\mathfrak{m}$. Then for any continuous function $f$ on $M$ we have the  \textit{co-area formula}
\[
\ds\int_Mf\, F(\nabla\varphi)\,d\mathfrak{m}=\int^{\infty}_{-\infty}\left( \ds\int_{\varphi^{-1}(t)}f\, dA\right)dt,\tag{2.1}\label{new2.1}
\]
see Shen \cite[Section 3.3]{Sh1}.
Define the divergence of a vector field $X$ by
\[
\di(X)\, d\mathfrak{m}:=d\left( X\rfloor d\mathfrak{m}\right).
\]
If $M$ is compact and oriented, we have the  divergence theorem
\[
\ds\int_M\di(X)d\mathfrak{m}=\ds\int_{\partial M} g_{\mathbf{n}}(\mathbf{n},X)\,d A,\tag{2.2}\label{new2.2}
\]
where $dA=\mathbf{n}\rfloor d\mathfrak{m}$, and $\mathbf{n}$ is the unit outward normal vector field along $\partial M$, i.e., $F(\mathbf{n})=1$ and $ g_{\mathbf{n}}(\mathbf{n},Y)=0$ for any $Y\in T(\partial M)$.

Given a $C^2$-function $f$, set $\mathcal {U}=\{x\in M:\, df|_x\neq0\}$. The \textit{Laplacian} of $f\in C^2(M)$ is defined on $\mathcal {U}$ by
\begin{align*}
\Delta f:=\text{div}(\nabla f)=\frac{1}{\sigma(x)}\frac{\partial}{\partial x^i}\left(\sigma(x)g^{*ij}(df|_x)\frac{\partial f}{\partial x^j}\right),\tag{2.3}\label{new2.3}
\end{align*}
where $(g^{*ij})$ is the fundamental tensor of $F^*$ and $x\mapsto \sigma(x)$ is the density function of  $d\mathfrak{m}$ in a local coordinate system $(x^i)$.  As in Ohta and Sturm \cite{Ot}, we define
the distributional Laplacian of $u\in W^{1,2}_{\text{loc}}(M)$
in the weak sense by
\[
\ds\int_M v{\Delta} u d\mathfrak{m}=-\ds\int_M\langle dv,\nabla u\rangle d\mathfrak{m} \text{ for all }v\in C^\infty_0(M),\tag{2.4}\label{new2.4}
\]
where $\langle dv,\nabla u\rangle= dv(\nabla u)$ at $x\in M$ denotes the canonical pairing between $T_x^*M$ and $T_xM.$

A smooth curve $t\mapsto \gamma(t)$ in $M$ is called a (constant speed) \textit{geodesic} if it satisfies
\[
\frac{d^2\gamma^i}{dt^2}+2G^i\left(\frac{d\gamma}{dt}\right)=0,
\]
where
\begin{align*}
G^i(y):=\frac14 g^{il}(y)\left\{2\frac{\partial g_{jl}}{\partial x^k}(y)-\frac{\partial g_{jk}}{\partial x^l}(y)\right\}y^jy^k
\end{align*}
is the geodesic coefficient.
In this paper, we always use $\gamma_y(t)$ to denote  the geodesic with $\dot{\gamma}_y(0)=y$.

$(M,F)$ is {\it forward complete} if  every geodesic $t\mapsto \gamma(t)$, $0\leq t<1$, can be extended to a geodesic defined on $0\leq t<\infty$; similarly,  $(M,F)$ is  {\it backward complete} if  every geodesic $t\mapsto \gamma(t)$, $0< t\leq 1$, can be extended to a geodesic defined on $-\infty< t\leq 1$.

\subsubsection{Curvatures} The {\it Riemannian curvature} $R_y$ of $F$ is a family of linear transformations on tangent spaces. More precisely, set
$R_y:=R^i_k(y)\frac{\partial}{\partial x^i}\otimes dx^k$, where
\[
R^i_{\,k}(y):=2\frac{\partial G^i}{\partial x^k}-y^j\frac{\partial^2G^i}{\partial x^j\partial y^k}+2G^j\frac{\partial^2 G^i}{\partial y^j \partial y^k}-\frac{\partial G^i}{\partial y^j}\frac{\partial G^j}{\partial y^k}.
\]
Let $P:=\text{Span}\{y,v\}\subset T_xM$ be a plane; the \textit{flag curvature} is defined by
\[
\mathbf{K}(y,v):=\frac{g_y\left( R_y(v),v  \right)}{g_y(y,y)g_y(v,v)-g^2_y(y,v)}.
\]
The  {\it Ricci curvature} of $y$
is defined by
\[
\mathbf{Ric}(y):=\underset{i}{\sum}\,\textbf{K}(y,e_i),
\]
where $e_1,\ldots, e_n$ is a $g_y$-orthonormal basis on $(x,y)\in
TM\backslash\{0\}$.

Let $\zeta:[0,1]\rightarrow M$ be a Lipschitz continuous path. The \textit{length} of $\zeta$ is defined by
\[
L_F(\zeta):=\int^1_0 F({\zeta}(t),\dot{\zeta}(t))dt.
\]
Define the distance function $d_F:M\times M\rightarrow [0,+\infty)$ by
$d_F(p,q):=\inf L_F(\zeta)$,
where the infimum is taken over all
Lipshitz continuous paths $\zeta:[0,1]\rightarrow M$ with
$\zeta(0)=p$ and $\zeta(1)=q$. Note that generally $d_F(p,q)\neq d_F(q,p)$,  unless $F$ is reversible.

Let $R>0$; the forward and backward metric balls $B^+_p(R)$ and $B^-_p(R)$ are defined by
\[
B^+_p(R):=\{x\in M:\, d_F(p,x)<R\},\ B^-_p(R):=\{x\in M:\, d_F(x,p)<R\}.
\]
If $F$ is reversible, forward and backward metric balls coincide which are denoted by $B_p(R)$.

Given $x_0\in M$, set $\rho_{x_0}(x):=d_F(x_0,x)$ and $\varrho_{x_0}(x):=d_F(x,x_0)$. In general, $\rho_{x_0}(x)\neq\varrho_{x_0}(x)$ unless $F(x_0,\cdot)$ is reversible, cf. \cite[Exercise 6.3.4]{BCS}. Moreover, one has by Shen \cite[Lemma 3.2.3]{Sh1} the eikoinal relations
\begin{equation}\label{eikonal}\tag{2.5}
F^*(d\rho_{x_0})=F(\nabla \rho_{x_0})=1,\ F^*(-d\varrho_{x_0})=F(\nabla (-\varrho_{x_0}))=1\ \ {\rm a.e.\ on}\ M.
\end{equation}


\subsubsection{Measures}

Let $d\mathfrak{m}$ be a measure on $M$; in a local coordinate system $(x^i)$ we
express $d\mathfrak{m}=\sigma(x)dx^1\wedge\ldots\wedge dx^n$. In particular,
the \textit{Busemann-Hausdorff measure} $d\mathfrak{m}_{BH}$ and the \textit{Holmes-Thompson measure} $d\mathfrak{m}_{HT}$ are defined by
\begin{align*}
&d\mathfrak{m}_{BH}:=\frac{\vol(\mathbb{B}^{n})}{\vol(B_xM)}dx^1\wedge\ldots\wedge dx^n,\\
 &d\mathfrak{m}_{HT}:=\left(\frac1{\vol(\mathbb{B}^{n})}\ds\int_{B_xM}\det g_{ij}(x,y)dy^1\wedge\ldots\wedge dy^n \right) dx^1\wedge\ldots\wedge dx^n,
\end{align*}
where $B_xM:=\{y\in T_xM: F(x,y)<1\}$ and $\mathbb{B}^{n}$ is the usual Euclidean $n$-dimensional unit ball.

Define the \textit{distortion} of $(M,F,d\mathfrak{m})$ as
\begin{equation*}
\tau(y):=\log \frac{\sqrt{\det g_{ij}(x,y)}}{\sigma(x)},\ \text{$y\in T_xM\backslash\{0\}$},
\end{equation*}
and the \textit{$S$-curvature} $\mathbf{S}$ given by
\begin{equation*}
\mathbf{S}(y):=\left.\frac{d}{dt}\right|_{t=0}[\tau(\dot{\gamma}_y(t))].
\end{equation*}

 The \textit{cut value} $i_y$ of $y\in S_xM$ is defined by
\[
i_y:=\sup\{r: \text{ the segment }\gamma_y|_{[0,r]} \text{ is globally minimizing}  \}.
\]
Hereafter, $S_xM:=\{y\in T_xM:F(x,y)=1\}$ and $SM:=\cup_{x\in M}S_xM$.
The \textit{injectivity radius} at $x$ is defined as $\mathfrak{i}_x:=\inf_{y\in S_xM} i_y$, whereas the \textit{cut locus} of $x$ is
\[
\text{Cut}_x:=\left\{\gamma_y(i_y):\,y\in S_xM \text{ with }i_y<\infty \right\}.
\]
Note that $\text{Cut}_x$ is closed and has null measure.

As in Zhao and Shen \cite{ZS}, if $x\in M$ is fixed, let $(r,y)$ be the polar coordinate system around $x$. Note that $r(w)=\rho_{x}(w)$ for any $w\in M$. Given an arbitrary measure $d\mathfrak{m}$, write
\[
d\mathfrak{m}:=\hat{\sigma}_x(r,y)dr\wedge d\nu_x(y),
\]
where $d\nu_x(y)$ is the Riemannian volume measure induced by $F$ on $S_xM$. Note that
\[
\lim_{r\rightarrow 0^+}\frac{\hat{\sigma}_x(r,y)}{r^{n-1}}=e^{-\tau(y)}.\tag{2.6}\label{new2.5}
\]

\subsubsection{Comparison principles.}
According to Zhao and Shen \cite[Theorems 3.4 \& 3.6, Remark 3.5]{ZS}, we have the following \textit{volume comparisons}:

\begin{itemize}
	\item[{\rm (i)}] If $\mathbf{K}\leq 0$ and $\mathbf{S}\leq 0$, for each $y\in S_xM$ we have
	\[
	\Delta r=\frac{\partial}{\partial r}\log \hat{\sigma}_x(r,y)\geq \frac{n-1}{r},\ 0<r< \mathfrak{i}_x.\tag{2.7}\label{new2.6}
	\]
	Hence,
	\[
	f(r):=\frac{\mathfrak{m}(B^+_x(r))}{\left(\ds\int_{S_xM}e^{-\tau(y)}d\nu_x(y)\right)\frac{r^n}{n}},\ 0<r< \mathfrak{i}_x,\tag{2.8}\label{new2.7}
	\]
	is non-decreasing and $f(r)\geq 1$,
	with equality for some $r_0>0$ if and only if $\mathbf{K}(\dot{\gamma}_y(t),\cdot)\equiv0$ and $\mathbf{S}(\dot{\gamma}_y(t))\equiv0$ for any $y\in S_xM$ and $0\leq t\leq r_0\leq \mathfrak{i}_x$.
	
	\item[{\rm (ii)}] If $\mathbf{Ric}\geq 0$ and $\mathbf{S}\geq 0$,
 for each $y\in S_xM$ we have
	\[
	\Delta r=\frac{\partial}{\partial r}\log \hat{\sigma}_x(r,y)\leq \frac{n-1}{r},\ 0<r< \mathfrak{i}_x.\tag{2.9}\label{new2.8}
	\]
	Therefore,
	\[
	f(r):=\frac{\mathfrak{m}(B^+_x(r))}{\left(\ds\int_{S_xM}e^{-\tau(y)}d\nu_x(y)\right)\frac{r^n}{n}},\ r>0,\tag{2.10}\label{new2.9}
	\]
	is non-increasing and $f(r)\leq 1$, with equality for some $r_0>0$ if and only if $\mathbf{K}(\dot{\gamma}_y(t),\cdot)\equiv0$ and $\mathbf{S}(\dot{\gamma}_y(t))\equiv0$ for any $y\in S_xM$ and $0\leq t\leq r_0\leq \mathfrak{i}_x$.
	
\end{itemize}

\subsubsection{Reversibility} The {\it reversibility} on $(M,F)$ is given by
\[
\lambda_F(M):=\sup_{x\in M}\lambda_F(x) \ \ {\rm with}\ \ \lambda_F(x)= \sup_{y\in T_xM\setminus\{0\}} \frac{F(x,-y)}{F(x,y)},
\]
 see Rademacher \cite{R}.   It is clear that $\lambda_F(M) =1$ if and only if $F $ is reversible.
Let $F^*$ be the dual Finsler metric of $F$. Set
\[
\lambda_{F^*}(x):=\sup_{\eta  \in T^*_xM\setminus\{0\}}\frac{F^*(x,-\eta)}{F^*(x,\eta)},\ \ \ \  \lambda_{F^*}(M):=\sup_{x\in M}\lambda_{F^*}(x).
\]
\begin{lemma}\label{reversLem}For each $x\in M$, one has
$\lambda_{F^*}(x)=\lambda_F(x)$ and hence, $\lambda_{F^*}(M)=\lambda_F(M)$.
\end{lemma}
\begin{proof} Fix any $\eta\in T^*_xM\setminus\{0\}$; thus for any $y\in T_xM\setminus\{0\}$ one has
\[
\frac{-\eta(y)}{F(x,y)}=\frac{\eta(-y)}{F(x,-y)}\frac{F(x,-y)}{F(x,y)}\leq F^*(x,\eta)\lambda_F(x).\] Therefore, $ F^*(x,-\eta)\leq F^*(x,\eta)\lambda_F(x),
$
which implies $\lambda_{F^*}(x)\leq\lambda_F(x)$.
Note that the Hahn-Banach
theorem implies   \[
F(x,y)=\sup_{\eta\in T^*_xM\backslash\{0\}}\frac{\eta(y)}{F^*(x,\eta)}.
\]
Using this fact and changing the roles of $F$ and $F^*$ in the above argument, one has $\lambda_{F}(x)\leq\lambda_{F^*}(x)$.
\end{proof}

\begin{lemma}\label{resib}
Let $(M,F)$ be a Finsler manifold,  $x\in M$ and set $S^*_xM:=\{\eta\in T^*_xM:\, F^*(x,\eta)=1\}$. Then the following statements are equivalent:

\begin{itemize}
	\item[{\rm (i)}] $F^*(x,\eta)\geq F^*(x,-\eta),\ \forall \eta\in S^*_xM;$
	\item[{\rm (ii)}] $F^*(x,-\eta)\geq F^*(x,\eta),\ \forall \eta\in S^*_xM;$
	\item[{\rm (iii)}] $F^*(x,\eta)=\lambda_F(x) F^*(x,-\eta),\ \forall \eta\in S^*_xM;$
	\item[{\rm (iv)}] $\lambda_F(x)=1$, i.e., $F(x,\cdot)$ is reversible.
\end{itemize}
\end{lemma}
\begin{proof}
(i)$\Rightarrow$(iv) Note that for every $\xi\in T^*_xM\setminus\{0\}$, one has $\eta:=\frac{\xi}{F^*(x,\xi)}\in S_x^*M.$ Applying property (i) for this element, if follows that
$F^*(x,\xi)\geq F^*(x,-\xi)$. Since $\xi$ is arbitrary, we may choose $\xi:=-\xi$ in the latter relation, which yields  $F^*(x,\xi)= F^*(x,-\xi)$, i.e., $\lambda_{F^*}(x)=1$. Now Lemma \ref{reversLem} provides $\lambda_F(x)=1$.

(ii)$\Rightarrow$(iv)  Applying (ii) for $\eta:=\frac{\xi}{F^*(x,\xi)}\in S_x^*M$ with $\xi\in T^*_xM\setminus\{0\}$, if follows that
$F^*(x,-\xi)\geq F^*(x,\xi)$. Since $\xi$ is arbitrary, we may choose again $\xi:=-\xi$,   which yields  $F^*(x,\xi)= F^*(x,-\xi)$, i.e., $\lambda_{F^*}(x)=1$.

(iii)$\Rightarrow$(iv) By  the positive homogeneity of $F^*(x,\cdot)$ and Lemma \ref{reversLem} we have that
$$1=\lambda_F(x)\sup_{\eta  \in S^*_xM}\frac{F^*(x,-\eta)}{F^*(x,\eta)}=\lambda_F(x)\sup_{\eta  \in T^*_xM\setminus\{0\}}\frac{F^*(x,-\eta)}{F^*(x,\eta)}=\lambda_F(x)\lambda_{F^*}(x)=\lambda_F^2(x).$$

(iv)$\Rightarrow$(i)$\&$(ii)\&(iii) Trivial.
\end{proof}

\subsubsection{Integral of distortion}
 Given two equivalent measures $d\mathfrak{m}_i$, $i=1,2$ on a Finsler manifold $(M,F)$  (i.e.,  there exits a constant $C>0$ such that $
d\mathfrak{m}_1=C d\mathfrak{m}_2
$), it is easy to see that the $S$-curvatures of these measures coincide.  In the sequel,
we denote by $[d\mathfrak{m}_{BH}]$ the equivalence class of the Busemann-Hausdorff measure.

\begin{definition}
Given a measure $d\mathfrak{m}$ on $(M,F)$, the integral of distortion is
\[
\mathscr{L}_{\mathfrak{m}}(x):=\frac{1}{n}\ds\int_{S_xM}e^{-\tau(y)}d\nu_x(y).
\]
\end{definition}

\begin{lemma}\label{lemm1}
Let  $d\mathfrak{m}$ be a measure on $(M,F)$. Then
\[
d\mathfrak{m}\in [d\mathfrak{m}_{BH}]\Longleftrightarrow \mathscr{L}_{\mathfrak{m}}\equiv\text{\rm constant.}
\]
\end{lemma}
\begin{proof}
	Given $x\in M$,
	let $(x^i)$ be a local coordinate system around $x$. If   $d\mathfrak{m}(x)=\sigma(x) dx^1\wedge\ldots \wedge dx^n$, one has
	\begin{eqnarray*}
	\mathscr{L}_{\mathfrak{m}}(x)&=&\frac{1}{n}\int_{S_xM}e^{-\tau(y)}d\nu_x(y)=\frac{1}{n}\int_{S_xM}\frac{\sigma(x)}{\sqrt{\det {g}_{ij}(x,y)}}d\nu_x(y)\\
	&=&\frac{1}{n}\int_{S_xM}\frac{\sigma(x)}{\sqrt{\det {g}_{ij}(x,y)}} \left(\sqrt{\det {g}_{ij}(x,y)}\sum_{i=1}^n(-1)^{i-1}  y^idy^1\wedge\ldots\wedge \widehat{dy^i}\wedge\ldots\wedge dy^n\right)\\
	&=&\frac{1}{n}\sigma(x)\vol(S_xM),
	\end{eqnarray*}
	where
	\[
	\vol(S_xM):=\int_{S_xM}\sum_{i=1}^n(-1)^{i-1}y^idy^1\wedge\ldots\wedge \widehat{dy^i}\wedge\ldots\wedge dy^n.
	\]
	Recall that $d\mathfrak{m}_{BH}(x)=\sigma_{BH}(x) dx^1\wedge\ldots \wedge dx^n$, where the density function is
	\[
	\sigma_{BH}(x)=\frac{\vol(\mathbb{B}^{n})}{\vol(B_xM)}=\frac{\vol(\mathbb{S}^{n-1})}{\vol(S_xM)}.
	\]
	The above computation yields that for some $C>0$,
$$
		\mathscr{L}_{\mathfrak{m}}\equiv C\Longleftrightarrow \sigma(x)=\frac{nC}{\vol(S_xM)}=\frac{nC}{\vol(\mathbb{S}^{n-1})}\sigma_{BH}(x),\  \forall x\in M
		\Longleftrightarrow d\mathfrak{m}=\frac{nC}{\vol(\mathbb{S}^{n-1})} d\mathfrak{m}_{BH},
$$
which concludes the proof.
\end{proof}

\begin{remark}\rm \label{remark-kell}
Note that $\vol(S_xM)$ depends on the choice of local coordinate system. Recall the 'natural/invariant' volume of $S_xM$ is given by
\[
\vol(x):=\ds\int_{S_xM}d\nu_x(y).
\]
\end{remark}

\begin{lemma}\label{lemm2}
Let $(M,F,d\mathfrak{m})$ be an $n$-dimensional forward complete Finsler manifold satisfying
$$
\mathbf{K}\leq 0,\ \mathbf{S}\leq 0\ {and}\  \mathfrak{i}_M=+\infty.
$$
 If there is some $x_0\in M$ with
\[
\mathscr{L}_{\mathfrak{m}}(x_0)=\inf_{x\in M}\mathscr{L}_{\mathfrak{m}}(x),\ \mathfrak{m}(B^+_{x_0}(r))=\mathscr{L}_{\mathfrak{m}}(x_0) r^n, \ \forall\,r>0,
\]
then $d\mathfrak{m}\in [ d\mathfrak{m}_{BH}]$, $\mathbf{K}=0$ and $\mathbf{S}=\mathbf{S}_{BH}=0$.
\end{lemma}

\begin{proof} Fix $x\in M$ arbitrarily.  According to (\ref{new2.7}),
 we have
$
\mathfrak{m}(B^+_{x}(r))\geq \mathscr{L}_{\mathfrak{m}}(x) r^n$ for every $r>0,
$
and
$
r\mapsto \frac{\mathfrak{m}(B^+_{x}(r))}{  r^n}$  is non-decreasing.
Thus, since $B^+_x(r)\subset B^+_{x_0}(r+d_F(x_0,x))$, we have
\begin{eqnarray*}
\mathscr{L}_{\mathfrak{m}}(x_0)&\leq& \mathscr{L}_{\mathfrak{m}}(x)\leq \frac{\mathfrak{m}(B^+_{x}(r))}{r^n}\\
&\leq &\underset{r\rightarrow +\infty}{\lim\sup}\frac{\mathfrak{m}(B^+_{x}(r))}{r^n}\leq \underset{r\rightarrow +\infty}{\lim\sup}\frac{\mathfrak{m}(B^+_{x_0}(r+d_F(x_0,x)))}{r^n}\\
&=&\underset{r\rightarrow +\infty}{\lim\sup}\left(\frac{\mathfrak{m}(B^+_{x_0}(r+d_F(x_0,x)))}{(r+d_F(x_0,x))^n} \frac{(r+d_F(x_0,x))^n}{r^n}  \right)\\
&=&\mathscr{L}_{\mathfrak{m}}(x_0),
\end{eqnarray*}
which implies that
\[
\mathscr{L}_{\mathfrak{m}}(x)=\mathscr{L}_{\mathfrak{m}}(x_0),\ \mathfrak{m}(B^+_{x}(r))=\mathscr{L}_{\mathfrak{m}}(x_0) r^n, \ \forall\,r>0.
\]
Therefore, by the equality case in the volume comparison principle it turns out that $\mathbf{K}\equiv0$ and $\mathbf{S}\equiv0$. Moreover, Lemma \ref{lemm1} implies that $d\mathfrak{m}\in [d\mathfrak{m}_{BH}] $ and hence $\mathbf{S}=\mathbf{S}_{BH}$.
\end{proof}

\begin{lemma}\label{lemm3}
Let $(M,F,d\mathfrak{m})$ be an $n$-dimensional forward complete Finsler manifold with
\begin{align*}
\mathbf{Ric}\geq 0,\ \mathbf{S}\geq0.
\end{align*}
 If there is some $x_0\in M$ with
\[
\mathscr{L}_{\mathfrak{m}}(x_0)=\sup_{x\in M}\mathscr{L}_{\mathfrak{m}}(x),\ \mathfrak{m}(B^+_{x_0}(r))=\mathscr{L}_{\mathfrak{m}}(x_0) r^n, \ \forall\,r>0,
\]
then
$d\mathfrak{m}\in [ d\mathfrak{m}_{BH}]$, $\mathbf{K}=0$ and $\mathbf{S}=\mathbf{S}_{BH}=0$.
\end{lemma}
\begin{proof} Fix $x\in M$ arbitrarily. Relation  (\ref{new2.9}) yields that
$
\mathfrak{m}(B^+_{x}(r))\leq \mathscr{L}_{\mathfrak{m}}(x) r^n$ for every $r>0,
$ and the function
$
r\mapsto \frac{\mathfrak{m}(B^+_{x}(r))}{  r^n}$  is non-increasing.
Since $B^+_x(r)\supset B^+_{x_0}(r-d_F(x,x_0))$ for sufficiently large $r>0$, it follows that
\begin{eqnarray*}
\mathscr{L}_{\mathfrak{m}}(x_0)&\geq& \mathscr{L}_{\mathfrak{m}}(x)\geq \frac{\mathfrak{m}(B^+_{x}(r))}{r^n}\\
&\geq &\underset{r\rightarrow +\infty}{\lim\sup}\frac{\mathfrak{m}(B^+_{x}(r))}{r^n}\geq \underset{r\rightarrow +\infty}{\lim\sup}\frac{\mathfrak{m}(B^+_{x_0}(r-d_F(x,x_0)))}{r^n}\\
&=&\underset{r\rightarrow +\infty}{\lim\sup}\left(\frac{\mathfrak{m}(B^+_{x_0}(r-d_F(x,x_0)))}{(r-d_F(x,x_0))^n} \frac{(r-d_F(x,x_0))^n}{r^n}  \right)\\
&=&\mathscr{L}_{\mathfrak{m}}(x_0).
\end{eqnarray*}
A similar argument as in the proof of Lemma \ref{lemm2} yields the required conclusion.
\end{proof}

\begin{remark}\rm The latter result above still holds if the assumptions
\[
\mathbf{Ric}\geq 0,\ \mathbf{S}\geq0,\ \mathfrak{m}(B^+_{x_0}(r))=\mathscr{L}_{\mathfrak{m}}(x_0) r^n,\ \forall\, r>0,\tag{C1}\label{Condition 1}
\]
 are replaced by
 \[
 \mathbf{Ric}_N\geq 0,\ \mathfrak{m}(B^+_{x_0}(r))=\mathscr{L}_{\mathfrak{m}}(x_0) r^N,\ \forall\, r>0,\tag{C2}\label{Condition 2}
 \]
 where $N\in[n,\infty]$.
  Indeed, by (\ref{new2.5}) we necessarily have $N=n$.
Furthermore, the definition of $\mathbf{Ric}_n$ (see e.g. Ohta and Sturm \cite{Ot} or Ohta \cite{O}) implies $ \mathbf{Ric}_n= \mathbf{Ric}$ and $\mathbf{S}\equiv0$. 
\end{remark}

\section{Heisenberg-Pauli-Weyl principle: Case (I) in (\ref{1.1})}\label{Sec4}

\subsection{Non-positively curved case (proof of Theorems \ref{fistCan}\&\ref{thirdCan} when $p=2$ and $q=0$)}
Let $(M,F,d \mathfrak{m})$ be an $n$-dimensional Cartan-Hadamard manifold (i.e., forward  complete simply connected Finsler manifold with non-positive flag curvature).
Given any $x\in M$, it is well-known that  there is no conjugate point to $x$  in $M$; therefore, each geodesic from $x$ is minimal and $\mathfrak{i}_x=+\infty$ for every $x\in M$.

\begin{proposition}\label{strongfirsthe-prop}
Let $(M,F, d\mathfrak{m})$ be an $n$-dimensional Cartan-Hadamard manifold with ${\bf S}\leq 0$ and let $J^{\rm max}_{2,0}$ be defined by {\rm (\ref{1.2})}. Let $x_0\in M$  be arbitrarily fixed.
Then we have the following:

\begin{itemize}
	\item[{\rm (i)}] $({ \mathbf{ J}}^{\rm max}_{2,0,x_0})$ holds, i.e.,  $J^{\rm max}_{2,0}(x_0,u)\geq \frac{n^2}{4}$ for any  $u\in C^\infty_0(M)\setminus \{0\}.$
	\item[{\rm (ii)}] $\frac{n^2}4$ is sharp in $({\mathbf{J}}^{\rm max}_{2,0,x_0})$ whenever $F^*(\mathcal {L}(\dot{\gamma}_y(t)))\geq F^*(-\mathcal {L}(\dot{\gamma}_y(t)))$ for any $y\in S_{x_0}M$ and $t\geq 0.$
	\item[{\rm (iii)}]  The following statements are equivalent:
	\begin{itemize}
		\item[{\rm (a)}] $\frac{n^2}{4}$ is achieved by an extremal in $({\mathbf{J}}^{\rm max}_{2,0,x_0});$
		\item[{\rm (b)}] $F^*(\mathcal {L}(\dot{\gamma}_y(t)))\geq F^*(-\mathcal {L}(\dot{\gamma}_y(t)))$, $\mathbf{K}(\dot{\gamma}_y(t),\cdot)\equiv0$ and $\mathbf{S}(\dot{\gamma}_y(t))\equiv0$ for all $ y\in S_{x_0}M$ and $t\geq 0$.
	\end{itemize}
\end{itemize}
\end{proposition}

\begin{proof} (i) Let us fix $u\in C_0^\infty(M)\setminus\{0\}$ arbitrarily.  Relation (\ref{new2.6}) together with the divergence theorem (\ref{new2.4}) yields
\begin{eqnarray}\label{3.00}
\nonumber\left( 2n\ds\int_M  u^2 d\mathfrak{m}  \right)^2&\leq&\left( 2\ds\int_M (1+\rho_{x_0}\Delta \rho_{x_0})u^2d\mathfrak{m} \right)^2\\
&=&\left( \ds\int_Mu^2\Delta\rho^2_{x_0}d\mathfrak{m} \right)^2=16\left(\ds\int_M u\rho_{x_0}\langle du,\nabla\rho_{x_0}    \rangle  d\mathfrak{m}\right)^2.
\end{eqnarray}
Set
\begin{align*}\left\{
\begin{array}{lll}
M_-:=\{x\in M:\, \langle du,\nabla\rho_{x_0}\rangle(x)<0\},\\
\\
M_+:=\{x\in M:\, \langle du,\nabla\rho_{x_0}\rangle(x)>0\},\tag{3.2}\label{3.0*}\\
\\
M_0:=\{x\in M:\, \langle du,\nabla\rho_{x_0}\rangle(x)=0\}.
\end{array}
\right.
\end{align*}
By the definition of the dual Finsler metric $F^*$, the eikoinal relation (\ref{eikonal}) and H\"older inequality we have
\begin{align*}
 \left|\ds\int_M u\rho_{x_0}\langle du,\nabla\rho_{x_0}    \rangle  d\mathfrak{m}\right| \leq &\ds\int_{M}\left|u\rho_{x_0}\langle du,\nabla\rho_{x_0}    \rangle\right|  d\mathfrak{m}\\
=&\ds\int_{M_-}|u|\rho_{x_0}\langle d(-u) ,\nabla\rho_{x_0}   \rangle  d\mathfrak{m}+\ds\int_{M_+}|u|\rho_{x_0}\langle du,\nabla\rho_{x_0}    \rangle   d\mathfrak{m}\tag{3.3}\label{3.6}\\
\leq &\ds\int_{M_-}|u|\rho_{x_0} F^*(-du)d\mathfrak{m}+\ds\int_{M_+}|u|\rho_{x_0} F^*(du) d\mathfrak{m}\tag{3.4}\label{3.7}\\
\leq &\ds\int_{M}|u|\rho_{x_0}\max\{F^*(\pm du)\}d\mathfrak{m}\tag{3.5}\label{new4.5**}\\
\leq &\left(\ds\int_{M}u^2\rho^2_{x_0}d\mathfrak{m} \right)^{\frac12}\left(\ds\int_{M}\max\{ F^{*2}(\pm du)\}d\mathfrak{m}\right)^\frac12, \tag{3.6}\label{3.5**}
\end{align*}
which together with (\ref{3.00}) implies that $J^{\rm max}_{2,0}(x_0,u)\geq \frac{n^2}{4}$.

\medskip

(ii) By assumption, we have $F^*(\mathcal {L}(\dot{\gamma}_y(t)))\geq F^*(-\mathcal {L}(\dot{\gamma}_y(t)))$ for any $y\in S_{x_0}M$ and $t\geq 0$. Set $$C_{HPW}:=\inf_{u\in C^\infty_0(M)\setminus\{0\}}J^{\rm max}_{2,0}(x_0,u),$$ i.e.,
for any $u\in C^\infty_0(M)$,
\[
\left(\ds\int_M\max \left\{F^{*2}(\pm du) \right\}d\mathfrak{m}\right)\left( \ds\int_M \rho^2_{x_0}  u^2 d\mathfrak{m} \right)\geq C_{HPW} \left(\ds\int_M  u^2 d\mathfrak{m}\right)^2.\tag{3.7}\label{3.500}
\]
By (i), one has $C_{HPW}\geq n^2/4$. Assume  by contradiction that $C_{HPW}>n^2/4$.
Let us choose a small $\delta>0$, consider the forward   ball $B^+_{x_0}(\delta)$ and let $(r,y)$ be the polar coordinate system around $x_0$. According to (\ref{new2.5}), there exists $\varepsilon=\varepsilon(\delta)>0$ such that $\lim_{\delta\rightarrow 0^+}\varepsilon(\delta)=0$ and
\begin{align*}
(1-\varepsilon(\delta))\cdot e^{-\tau(y)}  r^{n-1}dr\wedge d\nu_{x_0}(y)\leq d\mathfrak{m}(r,y)\leq (1+\varepsilon(\delta))\cdot e^{-\tau(y)}  r^{n-1}dr\wedge d\nu_{x_0}(y),
\end{align*}
for any  $(r,y)\in  B^+_{x_0}(\delta)\backslash\{x_0\}$.
For any $f\in C^\infty_0(B^+_{x_0}(\delta))$, inequality (\ref{3.500}) yields
\begin{align*}
&\left(\ds\int_{S_{x_0}M}e^{-\tau(y)}d\nu_{x_0}(y)\ds\int_0^\delta\max\left\{ F^{*2}(\pm df)\right\} r^{n-1}dr\right) \left( \ds\int_{S_{x_0}M}e^{-\tau(y)}d\nu_{x_0}(y)\int^\delta_0 f^2r^{n+1}dr \right)\\
& \ \ \ \ \ \ \ \ \ \ \ \ \ \ \ \ \geq C'_{HPW}\left( \ds\int_{S_{x_0}M}e^{-\tau(y)}d\nu_{x_0}(y)\int^\delta_0 f^2r^{n-1}dr \right)^2,
\end{align*}
where
\[
C'_{HPW}=C_{HPW}\left( \frac{1-\varepsilon(\delta)}{1+\varepsilon(\delta)} \right)^2>\frac{n^2}4
\]
for sufficiently small $\delta>0.$
For any $u\in C_0^\infty(M)$, choose an enough large $R>0$ such that
\[
f(r,y):=u\left(  {R} r,y \right)\in C^\infty_0(B^+_{x_0}(\delta))\cap C_0(B^+_{x_0}(\delta)).
\]
Then the above inequality reduces to
\begin{align*}
&\left(\ds\int_{S_{x_0}M}e^{-\tau(y)}d\nu_{x_0}(y)\ds\int_0^\infty \max\{F^{*2}(\pm du)\}r^{n-1}dr \right)\left( \ds\int_{S_{x_0}M}e^{-\tau(y)}d\nu_{x_0}(y)\ds\int^\infty_0 u^2r^{n+1}dr \right)\\
& \ \ \ \ \ \ \ \ \ \ \ \ \ \ \ \ \geq  C'_{HPW}\left( \ds\int_{S_{x_0}M}e^{-\tau(y)}d\nu_{x_0}(y)\ds\int^\infty_0 u^2r^{n-1}dr \right)^2.
\end{align*}
We now consider the test-function $u=e^{-r^2}$ which can be approximated by functions in $C^\infty_0(M)$. Recall that $r=\rho_{x_0}(r,y)$ and hence,
\[
dr=d\rho_{x_0}=\mathcal {L}(\nabla \rho_{x_0} )=\mathcal {L}\left((\exp_{x_0})_{*ry}y  \right)=\mathcal {L}\left(\dot{\gamma}_y(r)\right).
\]
Thus, for any $r>0$ and $y\in S_{x_0}M$, we have
\[
F^*(\mathcal {L}(\dot{\gamma}_y(r)))\geq F^*(-\mathcal {L}(\dot{\gamma}_y(r))) \Longleftrightarrow F^*\left(dr|_{(r,y)} \right)\geq F^*\left(-dr|_{(r,y)}\right),\tag{3.8}\label{new4.7}
\]
which implies $\max\{F^{*2}(\pm du)\}=4r^2e^{-2r^2}$. Hence, a direct calculation yields
\begin{align*}
\frac{n^2}{4}
\geq C'_{HPW}>\frac{n^2}4,
\end{align*}
which is a contradiction; accordingly, $C_{HPW}=n^2/4$.

(iii)  (a)$\Rightarrow$(b).
If there exists some extremal $u\in C^\infty(M)\backslash\{0\}$ in $({\mathbf{J}}^{\rm max}_{2,0,x_0})$ with $\ds\int_M u^2 d\mathfrak{m}<+\infty$, then (\ref{3.6}) implies that
\begin{itemize}
	\item[($\tilde a$)]  either $u\leq  0$ on $M_-$ and $u\geq 0$ on $M_+$;
	\item[($\tilde b$)]  or $u\geq  0$ on $M_-$ and $u\leq 0$ on $M_+$.
\end{itemize}
%
%
%
Moreover, the equality in (\ref{new4.5**}) implies that $u$ is constant on $M_0$ whenever $\mathfrak{m}(M_0)\neq 0$.
Let $(r,y)$ be the polar coordinate system around $x_0$. In particular, the equality in (\ref{3.7}) implies that $u=u(\rho_{x_0})=u(r)$. Thus, it follows that
$\frac{\partial u}{\partial r}< 0$ on $M_-$ and $\frac{\partial u}{\partial r}> 0$ on $M_+$, respectively.

In the case ($\tilde a$), the latter relations show that we necessarily have $\ds\int_M u^2d\mathfrak{m}=\infty$, a contradiction.
In the case ($\tilde b$), it turns out that
 $u$ is either nonpositive or nonnegative.
By the equality in (\ref{3.7}) and the H\"older inequality (\ref{3.5**}), we have $u\frac{\partial u}{\partial r}\leq 0$ and $\kappa r|u|=\left| \frac{\partial u}{\partial r}  \right|$ on $(0,\infty)$ for some  $\kappa>0$.
By these equations it follows that  $u=C e^{-\frac{\kappa}{2}r^2}$, where $C\in \mathbb{R}\backslash\{0\}$ and $\kappa>0$.

For convenience we may assume that $C>0$ and $\kappa=2$ (the case $C<0$ treats similarly). Since $\frac{\partial u}{\partial r}\leq 0$, it turns out that $\mathfrak{m}(M_+)= 0$ and  (\ref{3.7}) and (\ref{new4.5**}) yields $F^*(-du)\geq F^*(du)$.
In view of (\ref{new4.7}), we get $F^*(\mathcal {L}(\dot{\gamma}_y(t)))\geq F^*(-\mathcal {L}(\dot{\gamma}_y(t)))$ for every $t>0$ and $y\in S_{x_0}M$.
  Since $u>0$,  the equality in (\ref{3.00}) yields $\rho_{x_0}\Delta \rho_{x_0}=n-1$, which implies the equality in (\ref{new2.7}). Thus, the volume comparison principle then yields  $\mathbf{K}(\dot{\gamma}_y(t),\cdot)\equiv0$ and $\mathbf{S}(\dot{\gamma}_y(t))\equiv0$ for every $ y\in S_{x_0}M$ and $t\geq 0$.

\medskip

 (b)$\Rightarrow$(a). Since $\mathbf{K}(\dot{\gamma}_y(t),\cdot)\equiv0$ and $\mathbf{S}(\dot{\gamma}_y(t))\equiv0$ for every $ y\in S_{x_0}M$ and $t\geq 0$, we have the representation
  \[
  d\mathfrak{m}(r,y)=e^{-\tau(y)}r^{n-1}dr\wedge d\nu_{x_0}(y),\ 0<r<+\infty,\ y\in S_{x_0}M.
  \]
  Consider $u:=-e^{-\rho^2_{x_0}}=-e^{-r^2}$. Again  (\ref{new4.7}) yields
  $\max\{F^{*2}(\pm du)\}=4r^2 e^{-2r^2}.$
  Thus, a direct calculation furnishes
  \begin{align*}
&\ds\int_M \max\{F^{*2}(\pm du)\}d\mathfrak{m}=4n\mathscr{L}_{\mathfrak{m}}(x_0) \ds\int_0^\infty   e^{-2r^2}r^{n+1}dr=4\ds\int_M \rho^2_{x_0} u^2d\mathfrak{m},\\
&\ds\int_Mu^2d\mathfrak{m}=n\mathscr{L}_{\mathfrak{m}}(x_0)\ds\int_0^\infty   e^{-2r^2}r^{n-1}dr,
\end{align*}
which implies equality in $({\mathbf{J}}^{\rm max}_{2,0,x_0})$ for $u=-e^{-\rho^2_{x_0}}$, thus  (a) holds.
\end{proof}

In the case of $p=2,q=0$, Theorem \ref{fistCan} directly  follows from Proposition \ref{strongfirsthe-prop} and the following result.

\begin{proposition}\label{seoncdimpr-p}
Let $(M,F, d\mathfrak{m})$ be an $n$-dimensional Cartan-Hadamard manifold with
$ \mathbf{S} \leq 0.
$
If there exists a point $x_0\in M$ such that
\[
\lambda_F(x_0)= \lambda_F(M),\ \mathscr{L}_{\mathfrak{m}}(x_0)=\inf_{x\in M}\mathscr{L}_{\mathfrak{m}}(x),
\]
then the following statements are equivalent:

\begin{itemize}
	\item[{\rm (a)}] $\frac{n^2}{4}$ is achieved by an extremal in $({\mathbf{J}}^{\rm max}_{2,0,x_0});$
	\item[{\rm (b)}] $\frac{n^2}{4}$ is achieved by an extremal in $({\mathbf{J}}^{\rm max}_{2,0,x})$ for every $x\in M;$
	\item[{\rm (c)}]  $(M,F,d\mathfrak{m})$ is reversible, $d\mathfrak{m}\in [d\mathfrak{m}_{BH}]$, ${\bf K}= 0$ and $\mathbf{S}=\mathbf{S}_{BH}=0$.
\end{itemize}
\end{proposition}

\begin{proof} (b)$\Rightarrow$(a) is trivial.
	
		(c)$\Rightarrow$(b)
Since $(M,F)$ is reversible, by Proposition \ref{strongfirsthe-prop}/(ii), the constant ${n^2}/4$ is sharp in $({\mathbf{J}}^{\rm max}_{2,0,x})$ for every $x\in M$. Moreover, by Proposition \ref{strongfirsthe-prop}/(iii), it turns out that ${n^2}/{4}$ is achieved by an extremal in $({\mathbf{J}}^{\rm max}_{2,0,x})$ for every $x\in M$.

 (a)$\Rightarrow$(c).
 By Proposition \ref{strongfirsthe-prop}/(iii), one has $F^*(\mathcal {L}(\dot{\gamma}_y(t)))\geq F^*(-\mathcal {L}(\dot{\gamma}_y(t)))$ for every $t>0$ and $y\in S_{x_0}M$. By letting $t\rightarrow 0^+$, we have
$F^*(\mathcal {L}(y))\geq  F^*(-\mathcal {L}(y))$ for all $y\in S_{x_0}M.
$ Since the latter inequality is equivalent to
$F^*(x_0,\eta)\geq   F^*(x_0,-\eta)$ for all $\eta\in S^*_{x_0}M$, Lemma \ref{resib} implies that $\lambda_F(x_0)=1.$
Consequently, by the hypothesis it follows that $\lambda_F(M)=1$, i.e., $F$ is reversible. Due to the proof of Proposition \ref{strongfirsthe-prop}, the extremal function in $({\mathbf{J}}^{\rm max}_{2,0,x_0})$  has the particular form  $u=e^{-\rho^2_{x_0}}>0$ (up to scalar multiplication). Thus, by (\ref{3.00}) we have that $\rho_{x_0}\Delta\rho_{x_0}=n-1$ for every $x\in M\backslash\{x_0\}$; the equality case in the volume comparison principle implies that
$
\mathfrak{m}(B^+_{x_0}(r))=\mathscr{L}_{\mathfrak{m}}(x_0) r^n$ for all $r>0.
$
Now the statement directly follows by Lemma \ref{lemm2}.
\end{proof}

In the case $p=2,q=0$,  Theorem \ref{thirdCan} is  a consequence of  the   following result.
\begin{proposition}\label{reverLemmaflag-prop}
Let $(M,F,  d\mathfrak{m})$ be an $n$-dimensional  Cartan-Hadamard Finsler manifold with
$
\mathbf{S} \leq 0$ and $\lambda_F(M)<+\infty.
$
Then
\begin{align*}
J_{2,0}(x,u)\geq \frac{n^2}{4\lambda_F^2(M)},\ \ \forall  x\in M,\ u\in C^\infty_0(M)\backslash\{0\}.\tag{3.9}\label{3.700}
\end{align*}
In addition, assume that there exists a point $x_0\in M$ such that
$
\mathscr{L}_\mathfrak{m}(x_0)= \inf_{x\in M}\mathscr{L}_\mathfrak{m}(x)
.$   Then the following statements are equivalent:
\begin{itemize}
	\item[{\rm (a)}] $\frac{n^2}{4\lambda_F^2(M)}$ is achieved by an extremal $u$ in $({\mathbf{J}}_{2,0,x_0});$
	\item[{\rm (b)}] $\lambda_F(M)=1$, $d\mathfrak{m}\in [d\mathfrak{m}_{BH}],$  $\mathbf{K}=0$,  $\mathbf{S}=\mathbf{S}_{BH}=0$ and $u=C e^{-\kappa\rho_{x_0}^2}$ for some  $C\in\mathbb R\setminus \{0\}$ and $\kappa>0$.
\end{itemize}
\end{proposition}
\begin{proof}
Let us fix $x_0\in M$ and $u\in C_0^\infty(M)\setminus\{0\}$ arbitrarily. As in (\ref{3.5**}), we have
\begin{align*}
\left( 2n\ds\int_M  u^2 d\mathfrak{m}  \right)^2\leq &\left(  \ds\int_M \Delta \rho^2_{x_0} u^2 d\mathfrak{m} \right)^2\\
\leq &16\left(\ds\int_{M}u^2\rho^2_{x_0}d\mathfrak{m} \right) \left(\ds\int_{M}\max\{ F^{*2}(\pm du)\}d\mathfrak{m}\right)\\
\leq &16\,\lambda_F^2(M)\left( \ds\int_M u^2\rho^2_{x_0} d\mathfrak{m}   \right)  \left( \ds\int_M F^{*2}(du) d\mathfrak{m} \right).\tag{3.10}\label{2.2}
\end{align*}
which yields the validity of (\ref{3.700}), being  equivalent to $({\mathbf{J}}_{2,0,x_0})$.

Now assume that there exists a point $x_0\in M$ such that
 $
 \mathscr{L}_\mathfrak{m}(x_0)= \inf_{x\in M}\mathscr{L}_\mathfrak{m}(x).
 $ The implication (b)$\Rightarrow$(a) follows by  Proposition \ref{seoncdimpr-p}.

 (a)$\Rightarrow$(b)
Suppose that $\frac{n^2}{4\lambda_F^2(M)}$ is achieved by an extremal $u$ in $({\mathbf{J}}_{2,0,x_0})$. Note that in order to prove (\ref{3.700}), we explored the estimates (\ref{3.6})-(\ref{3.5**}) from Proposition \ref{strongfirsthe-prop}; hence, from its proof  we conclude that  $u=u(r)$ together with
$u\geq 0$ and $\frac{\partial u}{\partial r}< 0$ on $M_-$, and $u\leq 0$ and $\frac{\partial u}{\partial r}> 0$ on $M_+$,
where $(r,y)$ is the polar coordinate system around  $x_0$. It is immediate that either $M=M_-\sqcup M_0$ or $M=M_+\sqcup M_0$.  In particular, by the above properties, it follows that $u(x_0)\neq 0$.

We  claim that $F$ is reversible.

\textit{Case 1}:  $M=M_-\sqcup M_0$. By relations (\ref{3.6})-(\ref{3.5**}) and (\ref{2.2}) one gets
\[
\left\{
\begin{array}{lll}
& u F^*(-du)=u\max\{F^*(\pm du) \}=\lambda_F(M) u F^*(du)\ {\rm on}\  M,\\
\tag{3.11}\label{4.11new**}\\
&\max\{F^*(\pm du) \}=\kappa  |u| r\ {\rm on}\  M,
\end{array}
\right.
\]
where $\kappa\geq 0$ is a constant. Clearly, $\kappa>0$ otherwise $u=0$.
Since $u(x_0)\neq 0$ (in fact, $u(x_0)>0$), there exists a small forward ball $B^+_{x_0}(\delta)$ such that $u|_{B^+_{x_0}(\delta)} >0$ and
$F^*(\pm du)|_{B^+_{x_0}(\delta)\setminus \{x_0\}}>0$ (cf. (\ref{4.11new**})). Therefore,  $B^+_{x_0}(\delta)\backslash\{x_0\}\subset M_-$.
In particular, relation (\ref{4.11new**})  implies that for every $x\in B^+_{x_0}(\delta)\backslash\{x_0\}$, one has
$
   F^*(dr)=\lambda_F(M) F^*(-dr).
$
Thus, by Lemma \ref{reversLem} and the latter relation we have
\[
    \lambda_F(x_0)F^*(-\eta) \geq F^*(\eta)=\lambda_F(M) F^*(-\eta),\ \forall\eta\in S^*_{x_0}M, \tag{3.12}\label{4.12new**}
\]
which implies that $\lambda_F(x_0)\geq \lambda_F(M).$ By definition, the converse inequality also holds, thus  $\lambda_F(x_0)=\lambda_F(M);$ in particular, by  (\ref{4.12new**}) and Lemma \ref{resib} one has $\lambda_F(x_0)=1$, thus $\lambda_F(M)=1$, i.e., $F$ is reversible.

\textit{Case 2}: $M=M_+\sqcup M_0$. In this case we have
 \[
u F^*(du)=u\max\{F^*(\pm du) \}=\lambda_F(M) u F^*(du)\ {\rm on}\  M.\tag{3.13}\label{4.13new**}
\]
A similar argument as above yields the existence of a small  forward ball $B^+_{x_0}(\delta)\backslash\{x_0\}\subset M_+$ such that $u|_{B^+_{x_0}(\delta)} <0$.  Then (\ref{4.13new**}) restricted to ${B^+_{x_0}(\delta)}$ furnishes $\lambda_F(M)=1$.

Since $F$ is reversible, it turns out that $J^{\rm max}_{2,0}(x,u)=J_{2,0}(x,u)$ for every $x\in M$ and $u\in C_0^\infty(M)\setminus\{0\}$; thus Proposition \ref{seoncdimpr-p} provides the required properties.
\end{proof}


We conclude this subsection by stating an uncertainty principle without reversibility. To do this, let $(M,F)$ be a  Cardan-Hardamard manifold and $x_0\in M$ be fixed.
If $u\in C^1(M)$, then $\langle du, \nabla \rho_{x_0}\rangle$ exists on every point of $M$ except $x_0$.
We
introduce the following notation: for any $x\in M\backslash\{x_0\}$,
\begin{align*}
\|d u\|_{x_0, F}(x):=\left\{
\begin{array}{lll}
& F^*(-du)(x), & \ \ \ \text{if }\langle du, \nabla \rho_{x_0}\rangle(x)<0, \\
\\
& F^*(du)(x), & \ \ \ \text{if } \langle du, \nabla \rho_{x_0}\rangle(x)>0,\\
\\
&\frac12\left[F^*(du)(x)+F^*(-du)(x)\right], & \ \ \ \text{if } \langle du, \nabla \rho_{x_0}\rangle(x)=0.
\end{array}
\right.
\end{align*}
By the polar coordinate system around $x_0$, one can easily show that $\|d u\|_{x_0, F}$ is continuous a.e. on $M$ and agrees with $F^*(du)$ whenever $F$ is reversible.
Then Theorem \ref{fistCan} has the following alternative (we state it for both cases (I) and (II) from (\ref{1.1})):
\begin{theorem} \label{theorem-without}
	Given $p,q,n$ as in $(\ref{1.1}),$ and
	let $(M,F, d\mathfrak{m})$ be an $n$-dimensional Cardan-Hardamard manifold with
	$\ \mathbf{S}\leq0.
$
	Set
	\[
	\mathscr{J}_{p,q}(x,u):=\frac{ \left( \ds\int_M  \|d u\|^2_{x, F}   d\mathfrak{m}\right)\left(\ds\int_M \frac{|u|^{2p-2} }{\rho_{x}^{2q-2}} d\mathfrak{m}\right)}{\left(\ds\int_M \frac{|u|^p}{\rho_{x}^q}d\mathfrak{m} \right)^2},\   x\in M,\ u\in C^\infty_0(M)\backslash\{0\}.
	\]
	Then the following statements hold:

\begin{itemize}
	\item[{\rm (i)}] For every $x\in M$, $$\mathscr{J}_{p,q}(x,u)\geq \frac{(n-q)^2}{p^2},\ \forall u\in C^\infty_0(M)\backslash\{0\}. \eqno{({{\mathfrak J}}_{p,q,x})} $$   Moreover, $\frac{(n-q)^2}{p^2}$ is sharp, i.e.,
$
	\inf_{u\in C^\infty_0(M)\setminus\{0\}}\mathscr{J}_{p,q}(x,u)=\frac{(n-q)^2}{p^2}.
$
	\item[{\rm (ii)}] 	If there exists some point $x_0\in M$ such that
$
	\mathscr{L}_{\mathfrak{m}}(x_0)= \inf_{x\in M}\mathscr{L}_{\mathfrak{m}}(x),
$
	then the following statements are equivalent:

\begin{itemize}
	\item[{\rm (a)}] $\frac{(n-q)^2}{p^2}$ is achieved by an extremal in $({\rm {\mathfrak J}}_{p,q,x_0});$
	\item[{\rm (b)}] $\frac{(n-q)^2}{p^2}$ is achieved by an extremal in $({\rm {\mathfrak J}}_{p,q,x})$ for any $x\in M;$
	\item[{\rm (c)}]$(M,F,d\mathfrak{m})$  satisfies $d\mathfrak{m}\in [d\mathfrak{m}_{BH}]$, $\mathbf{K}=0$  and $\mathbf{S}=\mathbf{S}_{BH}=0$.
\end{itemize}
 \end{itemize}
\end{theorem}

\begin{proof} The proof is almost the same as before; we only consider the case (I) in (\ref{1.1}). Fix a point $x_0\in M$. By (\ref{3.6}), we have
	\begin{align*}
	\left|\ds\int_M u\rho_{x_0}\langle du, \nabla \rho_{x_0}    \rangle  d\mathfrak{m}\right|
	\leq &\ds\int_{M_-}|u|\rho_{x_0}\langle d(-u),\nabla\rho_{x_0}    \rangle  d\mathfrak{m}+\ds\int_{M_+}|u|\rho_{x_0}\langle du, \nabla \rho_{x_0}    \rangle   d\mathfrak{m} \\
	\leq &\ds\int_{M_-}|u|\rho_{x_0}\|du\|_{x_0,F}d\mathfrak{m}+\ds\int_{M_+}|u|\rho_{x_0}\|du\|_{x_0,F} d\mathfrak{m}\\
	=&\ds\int_{M}|u|\rho_{x_0}\|du\|_{x_0,F}d\mathfrak{m}\leq \left(\ds\int_M|u|^2\rho^2_{x_0}d\mathfrak{m}  \right)^\frac12\left(\ds\int_M  \|du\|^2_{x_0,F}d\mathfrak{m}\right)^\frac12,
	\end{align*}
	which together with (\ref{3.00}) yields the required inequality. The rest of the proof is the same as in Theorem \ref{fistCan}. 
\end{proof}

\subsection{Non-negatively curved case (proof of Theorem \ref{forCan} when $p=2$ and $q=0$)}

In the case  $p=2$ and $q=0$,  Theorem \ref{forCan} directly follows by the following result.
\begin{proposition}\label{Ricclemm-p}
Let $(M,F, d\mathfrak{m})$ be an $n$-dimensional  forward  complete Finsler manifold with
$\mathbf{Ric}\geq 0$ and $\mathbf{S}\geq 0.
$
If for some $x_0\in M$,
\[
\lambda_F(x_0)= \lambda_F(M),\ \mathscr{L}_{\mathfrak{m}}(x_0)=\sup_{x\in M} \mathscr{L}_{\mathfrak{m}}(x),
\]
then
\begin{align*}
\left(\ds\int_M \min\left\{ F^{*2}(\pm du)  \right\}d \mathfrak{m} \right)\left( \ds\int_M \rho^2_{x_0}  u^2 d\mathfrak{m} \right)\geq \frac{n^2}{4}\left(\ds\int_M u^2 d\mathfrak{m}\right)^2\tag{3.14}\label{3.12}
\end{align*}
holds for any $u\in C^\infty_0(M)$ if and only if
$
\lambda_F(M)=1,\  d\mathfrak{m}\in [d\mathfrak{m}_{BH}],\ \mathbf{K}=0$ and $ \mathbf{S}=\mathbf{S}_{BH}=0.
$
\end{proposition}
\begin{proof} The "if" part is trivial; in the sequel, we deal with the "only if" part.
First, we observe that $M$ is not compact. Now, consider $u_s(x)=e^{-s\rho^2_{x_0}(x)}$ for $s>0$, $x\in M$. A direct calculation yields
\begin{align*}
\ds\int_M \min\left\{ F^{*2}(du_s),\, F^{*2}(-du_s) \right\}d \mathfrak{m} \leq 4s^2\ds\int_M \rho^2_{x_0}  u_s^2 d\mathfrak{m}.
\end{align*}
Hence, putting as the test-function $u_s$ in (\ref{3.12}), it follows that
\begin{align*}
 {2s}\ds\int_M \rho^2_{x_0} e^{-2s{\rho^2_{x_0}}} d \mathfrak{m}\geq \frac{n}{2}\ds\int_Me^{-2s{\rho^2_{x_0}}} d \mathfrak{m}, \ s>0.\tag{3.15}\label{3.13}
\end{align*}
Now set
\[
\mathscr{T}(s):=\ds\int_M   u_s^2 d\mathfrak{m}=\ds\int_Me^{-2s{\rho^2_{x_0}}} d \mathfrak{m}.
\]
The layer cake representation yields
\begin{align*}
\mathscr{T}(s)=&\ds\int_0^\infty \mathfrak{m}\left(\left\{ x\in M:\, e^{-2s{\rho^2_{x_0}(x)}}>t \right\} \right)dt\\ =&\ds\int_0^1 \mathfrak{m}\left(\left\{ x\in M:\, e^{-2s{\rho^2_{x_0}(x)}}>t \right\} \right)dt\\
 =&4s\int^\infty_0 le^{-2s {l^2} }\mathfrak{m}\left(\left\{ x\in M:\, {\rho_{x_0}(x)}<l \right\} \right)dl\\
 =& {4s} \int^\infty_0 le^{-2s {l^2} }\mathfrak{m}\left(B^+_{x_0}(l) \right)dl.\tag{3.16}\label{3.14}
\end{align*}
Since $\mathbf{Ric}\geq 0$ and $\mathbf{S}\geq 0$, (\ref{new2.9}) furnishes
\begin{equation}\label{bg-nov}\tag{3.17}
\mathfrak{m}\left(B^+_{x_0}(l) \right)\leq \mathscr{L}_{\mathfrak{m}}(x_0)\, l^n, \ \forall\,l>0,
\end{equation}
and hence,
\begin{align*}
\mathscr{T}(s)\leq  {4s} \mathscr{L}_{\mathfrak{m}}(x_0)\int^\infty_0  l^{n+1}e^{-2s {l^2} } dl<+\infty.
\end{align*}
In particular, $\mathscr{T}$ is well-defined and (\ref{3.13}) can be equivalently transformed into
\[
-s \mathscr{T}'(s)\geq \frac{n}{2 }\mathscr{T}(s), \ \forall s>0,
\]
which implies
\[
\frac{\mathscr{T}'(s)}{\mathscr{T}(s)}\leq -\frac{n}{2s }= \frac{T'(  s)}{T(  s)},
\]
where
\[
T(s):=\frac{\mathscr{L}_{\mathfrak{m}}(x_0)}{\vol(\mathbb{B}^{n})}\ds\int_{\mathbb{R}^n} e^{-2s|x|^2}dx,\ \  \ -s T'(s)=\frac{n}2T(s).
\]
Then we obtain that
\[
\frac{d}{ds}\ln\left[ \frac{\mathscr{T}(s)}{T(s)} \right]\leq0\Longrightarrow f(s):= \frac{\mathscr{T}(s)}{T(s)}\text{ is non-increasing on}\ (0,\infty).
\]
Therefore, for every $s\in (0,\infty)$,
\[
f(s)\geq \underset{s\rightarrow +\infty}{\lim\inf}f(s).\tag{3.18}\label{3.15}
\]
Note that (\ref{new2.5}) implies
\[
\lim_{r\rightarrow 0^+}\frac{\mathfrak{m}(B^+_{x_0}(r))}{\mathscr{L}_\mathfrak{m}(x_0) r^n}= 1.
\]
Hence, for any $\varepsilon>0$, there exists $r_\varepsilon>0$ such that for any $r\in (0,r_\varepsilon)$,
$
\mathfrak{m}(B^+_{x_0}(r)) \geq (1-\varepsilon)\mathscr{L}_\mathfrak{m}(x_0) r^n,
$
which together with (\ref{3.14}) yields
\begin{align*}
\mathscr{T}(s)&\geq 4s\int^{r_\varepsilon}_0 t e^{-2st^2}\mathfrak{m}\left(B^+_{x_0}(t)\right)dt\geq 4s(1-\varepsilon)\mathscr{L}_{\mathfrak{m}}(x_0)\int^{r_\varepsilon}_0 t^{n+1}e^{-2st^2}dt\\
&\geq \frac{2}{(2s)^\frac{n}2}(1-\varepsilon)\mathscr{L}_{\mathfrak{m}}(x_0)\int^{\sqrt{2s}r_\varepsilon}_0 t^{n+1}e^{-t^2}dt.\tag{3.19}\label{3.16}
\end{align*}
Since
\[
T(s)= \frac{2}{(2s)^\frac{n}2} \mathscr{L}_{\mathfrak{m}}(x_0)\int^{\infty}_0 t^{n+1}e^{-t^2}dt,
\]
relation (\ref{3.16}) implies that
\[
\underset{s\rightarrow +\infty}{\lim\inf} \frac{\mathscr{T}(s)}{T(s)}\geq 1-\varepsilon.
\]
The arbitrariness of $\varepsilon>0$ together with (\ref{3.15}) yields
\[
\mathscr{T}(s) \geq T(s), \ \forall s>0,
\]
i.e., by  (\ref{3.14}),
\begin{align*}
\int^\infty_0t e^{-2st^2}\left( \mathfrak{m}(B^+_{x_0}(t))-\mathscr{L}_{\mathfrak{m}}(x_0) t^n  \right)dt\geq 0.
\end{align*}
On the other hand, relation (\ref{bg-nov}) together with the latter relation implies
\[
\mathfrak{m}(B^+_{x_0}(t))= \mathscr{L}_{\mathfrak{m}}(x_0) t^n,\ \forall t>0.
\]
 Now it follows from Lemma \ref{lemm3} that $d\mathfrak{m}\in [ d\mathfrak{m}_{BH}]$, $\mathbf{K}=0$ and $\mathbf{S}=\mathbf{S}_{BH}=0$.

It remains to prove the reversibility of $F$. To do this, let $(r,y)$ be the polar coordinate system around  $x_0$ and let $u=e^{-\rho_{x_0}^2}$ be the test function in (\ref{3.12}), i.e.,
\[
\left(\int^\infty_0 4e^{-2r^2}\min\{1,\, F^{*2}(-dr)\}r^{n+1}dr \right)\left( \int^\infty_0e^{-2r^2}r^{n+1}dr\right)\geq \frac{n^2}4\left( \int^\infty_0 e^{-2r^2}r^{n-1}dr \right)^2,
\]
which is nothing but
\[
\int^\infty_0e^{-2r^2}\min\{1, F^{*2}(-dr)\}r^{n+1}dr\geq \int^\infty_0e^{-2r^2}r^{n+1}dr.
\]
Therefore, we necessarily have $\min\{1, F^{*2}(-dr)\}=1$ for any $(r,y)\in M$,  i.e.,  $1=F^*(dr)\leq F^*(-dr)$. In particular, it turns out that $F^*(x_0,\eta)\leq F^*(x_0,-\eta)$ for every $\eta\in S_{x_0}^*M$.  Now Lemma \ref{resib} implies $\lambda_{F}(x_0)=1$; by  the assumption $\lambda_F(x_0)= \lambda_F(M)$ we conclude the proof.
\end{proof}

\section{Caffarelli-Kohn-Nirenberg interpolation inequality: Case (II) in (\ref{1.1})}\label{Sec5}

In this section we shortly present the proof of Theorems \ref{fistCan}-\ref{forCan} in the case (II) of (\ref{1.1}). Since the arguments are similar to those in the previous section, we focus only on  the differences.

\subsection{Non-positively curved case (proof of Theorems \ref{fistCan}\&\ref{thirdCan} when $0<q<2<p$ and $2<n<\frac{2(p-q)}{p-2}$)}

The counterpart of Proposition \ref{strongfirsthe-prop} reads as follows.

\begin{proposition}\label{CKNFirst}
Let $(M,F, d\mathfrak{m})$ be an $n$-dimensional Cartan-Hadamard manifold with ${\bf S}\leq 0$ and let $J^{\rm max}_{p,q}$ be defined by $(\ref{1.2})$ with $p,q\in \mathbb{R}$ and $n\in \mathbb{N}$ as in the case {\rm  (II)} of $(\ref{1.1}).$ Let $x_0\in M$ be arbitrarily fixed. Then we have the following:

\begin{itemize}
	\item[{\rm (i)}]  $({ \mathbf{ J}}^{\rm max}_{p,q,x_0})$ holds, i.e., $J^{\rm max}_{p,q}(x_0,u)\geq \frac{(n-q)^2}{p^2}$ for every 
	$u\in C^\infty_0(M)\setminus \{0\}.$
	\item[{\rm (ii)}] $\frac{(n-q)^2}{p^2}$ is sharp in $({ \mathbf{ J}}^{\rm max}_{p,q,x_0})$ whenever   $F^*(\mathcal {L}(\dot{\gamma}_y(t)))\geq F^*(-\mathcal {L}(\dot{\gamma}_y(t)))$ for any $y\in S_{x_0}M$ and $t\geq 0.$
	\item[{\rm (iii)}]  The following statements are equivalent:
	\begin{itemize}
		\item[{\rm (a)}] $\frac{(n-q)^2}{p^2}$ is achieved by an extremal in $({ \mathbf{ J}}^{\rm max}_{p,q,x_0});$
		\item[{\rm (b)}]   $F^*(\mathcal {L}(\dot{\gamma}_y(t)))\geq F^*(-\mathcal {L}(\dot{\gamma}_y(t)))$,  $\mathbf{K}(\dot{\gamma}_y(t),\cdot)\equiv0$ and $\mathbf{S}(\dot{\gamma}_y(t))\equiv0$ for all $ y\in S_{x_0}M$ and $t\geq 0$.
	\end{itemize}
\end{itemize}
\end{proposition}

\begin{proof}
(i) Fix  $u\in C^\infty_0(M)$ arbitrarily; then we have
\begin{align*}
\ds\int_M \frac{|u|^p}{\rho_{x_0}^{q-1}}\Delta \rho_{x_0} d\mathfrak{m}
=&-\ds\int_M\left\langle d\left( \frac{|u|^p}{\rho^{q-1}_{x_0}}\right),  \nabla\rho_{x_0}     \right\rangle d\mathfrak{m}\\
=&-p\ds\int_M \frac{|u|^{p-2}u}{\rho_{x_0}^{q-1}}\left\langle du,\nabla\rho_{x_0}    \right\rangle d\mathfrak{m}+(q-1)\ds\int_M \frac{|u|^p}{\rho_{x_0}^q}d\mathfrak{m}\\
\leq& p\left|\ds\int_M \frac{|u|^{p-2}u}{\rho_{x_0}^{q-1}}\left\langle du,\nabla\rho_{x_0}  \right\rangle d\mathfrak{m}   \right|+(q-1)\ds\int_M \frac{|u|^p}{\rho_{x_0}^q}d\mathfrak{m}.\tag{4.1}\label{5.1}
\end{align*}
Set $M_-$, $M_+$ and $M_0$ as in (\ref{3.0*}).
Then one has that
\begin{align*}
\left|\ds\int_M \frac{|u|^{p-2}u}{\rho_{x_0}^{q-1}}\left\langle du,\nabla\rho_{x_0}   \right\rangle d\mathfrak{m}   \right|
= &\left|\ds\int_{M_-} \frac{|u|^{p-2}(-u)}{\rho_{x_0}^{q-1}}\left\langle d(-u),\nabla\rho_{x_0}    \right\rangle d\mathfrak{m}  + \ds\int_{M_+} \frac{|u|^{p-2}u}{\rho_{x_0}^{q-1}}\left\langle du,\nabla\rho_{x_0}  \right\rangle d\mathfrak{m}\right|\\
\leq &\left|\ds\int_{M_-} \frac{|u|^{p-2}(-u)}{\rho_{x_0}^{q-1}}\left\langle d(-u),\nabla\rho_{x_0}    \right\rangle d\mathfrak{m} \right| + \left|\ds\int_{M_+} \frac{|u|^{p-2}u}{\rho_{x_0}^{q-1}}\left\langle du,\nabla\rho_{x_0}   \right\rangle d\mathfrak{m}\right|\\
\leq &\ds\int_{M_-} \frac{|u|^{p-1} }{\rho_{x_0}^{q-1}}\left\langle d(-u),\nabla\rho_{x_0}    \right\rangle d\mathfrak{m}+\ds\int_{M_+} \frac{|u|^{p-1} }{\rho_{x_0}^{q-1}}\left\langle du,\nabla\rho_{x_0}   \right\rangle d\mathfrak{m}\tag{4.2}\label{5.2}\\
\leq &\ds\int_{M_-} \frac{|u|^{p-1} }{\rho_{x_0}^{q-1}}F^*(-du)    d\mathfrak{m}+\ds\int_{M_+} \frac{|u|^{p-1} }{\rho_{x_0}^{q-1}}F^*(du) d\mathfrak{m}\tag{4.3}\label{5.3}\\
\leq& \ds\int_{M} \frac{|u|^{p-1} }{\rho_{x_0}^{q-1}}\max\left\{ F^*(\pm du)  \right\}  d\mathfrak{m}\\
\leq &\left(\ds\int_M \frac{|u|^{2p-2} }{\rho_{x_0}^{2q-2}} d\mathfrak{m}\right)^{\frac12}\left( \ds\int_M  \max\left\{ F^{*2}(\pm du)  \right\}   d\mathfrak{m}\right)^{\frac12},
\end{align*}
which together with (\ref{5.1}) and the Laplace comparison (\ref{new2.6}) yield
\begin{align*}
\left(\ds\int_M \frac{|u|^{2p-2} }{\rho_{x_0}^{2q-2}} d\mathfrak{m}\right) \left( \ds\int_M  \max\left\{ F^{*2}(\pm du)  \right\}    d\mathfrak{m}\right)\geq \frac{(n-q)^2}{p^2}\left(\ds\int_M \frac{|u|^p}{\rho_{x_0}^q}d\mathfrak{m} \right)^2.
\end{align*}

(ii) The sharpness of the constant $\frac{(n-q)^2}{p^2}$ follows in a similar way as in Proposition \ref{strongfirsthe-prop}/(ii); the only difference in the last step is the use of the test function $u=(r^{2-q}+1)^{\frac1{2-p}}$ instead of $u=e^{-r^2}$.

(iii)  Let  $(r,y)$ be the polar coordinate system about $x_0$. If $u$ is an extremal in $({ \mathbf{ J}}^{\rm max}_{p,q,x_0})$, then (\ref{5.3}) implies $u=u(\rho_{x_0})=u(r)$. By the equalities in (\ref{5.1})-(\ref{5.3}) and H\"older inequality, a similar argument as in
 Proposition \ref{strongfirsthe-prop} implies $u\frac{\partial u}{\partial r}\leq 0$ and $\kappa   \frac{|u|^{p-1}}{r^{q-1}}=\left| \frac{\partial u}{\partial r}  \right|$ on $(0,\infty)$ for some
 $\kappa> 0$. By solving this ODE, it follows that $
 u=C_1(r^{2-q}+C_2)^{\frac1{2-p}},$
 for some $C_1\in \mathbb{R} $ and $C_2>0$. In particular,  $u$ has no zero points. The rest of the proof is similar to the one of Proposition \ref{strongfirsthe-prop}/(iii).
 \end{proof}


In the case (II) of (\ref{1.1}), Theorem \ref{fistCan} directly follows from Proposition \ref{strongfirsthe-prop} and the following result; since the proof is almost the same as Proposition \ref{seoncdimpr-p}, we omit it.

\begin{proposition}\label{CKNtheorem}
Under the same assumptions as in Proposition \ref{CKNFirst}, if there exists some point $x_0\in M$ such that
\[
\lambda_F(x_0)=\lambda_F(M),\ \mathscr{L}_{\mathfrak{m}}(x_0)= \inf_{x\in M}\mathscr{L}_{\mathfrak{m}}(x),
\]
then the following statements are equivalent:
\begin{itemize}
	\item[{\rm (a)}] $\frac{(n-q)^2}{p^2}$ is achieved by an extremal in  $({ \mathbf{ J}}^{\rm max}_{p,q,x_0});$
	\item[{\rm (b)}] $\frac{(n-q)^2}{p^2}$ is achieved by an extremal in  $({ \mathbf{ J}}^{\rm max}_{p,q,x})$ for every $x\in M;$
	\item[{\rm (c)}]
	$(M,F,d\mathfrak{m})$ is reversible, $d\mathfrak{m}\in [d\mathfrak{m}_{BH}]$, ${\bf K}=0$ and $\mathbf{S}=\mathbf{S}_{BH}=0$.
\end{itemize}
\end{proposition}

By a similar argument as in Proposition \ref{reverLemmaflag-prop} one can easily show the following result which implies Theorem \ref{thirdCan} in the case {\rm  (II)} of (\ref{1.1}).
\begin{proposition}
	Let $(M,F,  d\mathfrak{m})$ be an $n$-dimensional  Cartan-Hadamard manifold with
	$
	\mathbf{S} \leq 0,$  $\lambda_F(M)<+\infty,
	$
and   $p,q\in \mathbb{R}$ and $n\in \mathbb{N}$  as in the case {\rm  (II)} of $(\ref{1.1}).$
	Then
	\begin{align*}
	J_{p,q}(x,u)\geq \frac{(n-q)^2}{p^2\lambda_F^2(M)},\ \forall  x\in M,\ u\in C^\infty_0(M)\backslash\{0\}.\tag{4.4}\label{3.7000}
	\end{align*}
In addition, assume that there exists a point $x_0\in M$ such that
$
\mathscr{L}_\mathfrak{m}(x_0)= \inf_{x\in M}\mathscr{L}_\mathfrak{m}(x)
.$ Then the following statements are equivalent:
\begin{itemize}
	\item[{\rm (a)}] $\frac{(n-q)^2}{p^2\lambda_F^2(M)}$ is achieved by an extremal $u$ in $({\mathbf{J}}_{p,q,x_0});$
	\item[{\rm (b)}]  $\lambda_F(M)=1$, $d\mathfrak{m}\in [d\mathfrak{m}_{BH}],$  $\mathbf{K}=0$,  $\mathbf{S}=\mathbf{S}_{BH}=0$ and $u=C_1(\rho_{x_0}^{2-q}+C_2)^{\frac1{2-p}}$ for some  $C_1\neq0$ and $C_2>0 $.
\end{itemize}
\end{proposition}

\begin{remark}\rm
The proof of Theorem \ref{theorem-without} in the case {\rm  (II)} of (\ref{1.1}) easily follows by the arguments performed in Propositions \ref{CKNFirst} and \ref{CKNtheorem}, respectively.
\end{remark}

\subsection{Non-negatively curved case (proof of Theorem \ref{forCan} when $0<q<2<p$ and $2<n<\frac{2(p-q)}{p-2}$)}

The proof of Theorem \ref{forCan} in the case {\rm  (II)} of (\ref{1.1})  directly follows by the following result.
\begin{proposition}
	Let $(M,F, d\mathfrak{m})$ be an $n$-dimensional  forward  complete Finsler manifold with
	$\mathbf{Ric}\geq 0,\ \mathbf{S}\geq 0,
	$ and  $p,q\in \mathbb{R}$ and $n\in \mathbb{N}$ as in the case {\rm  (II)} of $(\ref{1.1}).$
	If for some $x_0\in M$,
	\[
	\lambda_F(x_0)= \lambda_F(M),\ \mathscr{L}_{\mathfrak{m}}(x_0)=\sup_{x\in M} \mathscr{L}_{\mathfrak{m}}(x),
	\]
	then
	\begin{align*}
	{\left(\ds\int_M \min \{F^{*2}(\pm du)\}  d\mathfrak{m} \right)\left( \ds\int_M \frac{|u|^{2p-2}}{\rho^{2q-2}_{x_0}}    d\mathfrak{m} \right)}\geq \frac{(n-q)^2}{p^2}{\left(\ds\int_M  \frac{|u|^p}{\rho^q_{x_0}} d\mathfrak{m}\right)^2}
	\end{align*}
	holds for every $u\in C^\infty_0(M)$ if and only if
	\[
	\lambda_F(M)=1,\  d\mathfrak{m}\in [d\mathfrak{m}_{BH}],\ \mathbf{K}=0,\ \mathbf{S}=\mathbf{S}_{BH}=0.
	\]
\end{proposition}

\begin{proof}
The proof is similar to that of Proposition \ref{Ricclemm-p}; the main difference is to use the test function $u_s(x)=(\rho_{x_0}^{2-q}+s)^{\frac1{2-p}}$ for   $s>0$ instead of $u_s(x)=e^{-s\rho^2_{x_0}(x)}$ for $s>0$. The case when $\lambda_F(M)=1$ and $\mathfrak{m}= \mathfrak{m}_{BH}$ has been considered by Krist\'aly \cite[Theorem 1.2]{Kristaly-JGA}.
\end{proof}

\medskip

\section{Hardy inequality (proof of Theorem \ref{Hardyineq})}\label{section-Hardy}
We first need the following technical lemma.
\begin{lemma}\label{implem}
Given $n\geq 2$, let $(M,F)$ be an $n$-dimensional forward or backward complete Finsler manifold. Then for any $x_0\in M$ and any $k\in [0,n)$, we have
\[
\ds\int_M \left|\frac{u(x)}{\rho^k_{x_0}(x)}\right|d\mathfrak{m}(x)<+\infty,\ \forall u\in C_0^\infty(M).
\]
 \end{lemma}
\begin{proof}
According to Yuan,  Zhao and Shen  \cite[Proposition 3.2]{YZY}, there is a polar coordinate   domain $\mathcal {O}\subset T_{x_0}M$ such that $\exp_{x_0}(\mathcal {O})=M$. Let $(r,y)$ be the polar coordinate system around $x_0$. Since $u\in C^\infty_0(M)$, there exists a finite $R>0$ such that $\text{supp}(u)\subset B^+_{x_0}(R)$ and $\exp_{x_0}: \mathfrak{B}^+_0(R)\rightarrow B^+_{x_0}(R)$ is a diffeomorphism, where
$
\mathfrak{B}^+_0(R):=\{y\in T_{x_0}M:\, F(x_0,y)<R\}\cap \mathcal {O}.
$
Now set $A:=\max|u|<+\infty$. Then we have
\begin{align*}
\ds\int_M \left|\frac{u(x)}{\rho^k_{x_0}(x)}\right|d\mathfrak{m}(x)\leq \ds\int_{B^+_{x_0}(R)}\frac{A}{\rho^k_{x_0}(x)}d\mathfrak{m}(x)=\ds\int_{S_{x_0}M}d\nu_{x_0}(y) \ds\int_{0}^{\min\{R,i_y\}}\frac{A}{r^k}\hat{\sigma}_{x_0}(r,y)dr.\tag{5.1}\label{5.5***}
\end{align*}
Now (\ref{new2.5}) yields that there is a small $\varepsilon>0$ such that $\min\{R,i_y\}>\varepsilon$ for all $y\in S_{x_0}M$ and
\[
\hat{\sigma}_{x_0}(r,y)<2e^{-\tau(y)}r^{n-1},\ \text{  $0<r<\varepsilon$};
\]
the latter relation  together with (\ref{5.5***}) and Remark \ref{remark-kell} furnishes
\begin{align*}
\ds\int_M \left|\frac{u(x)}{\rho^k_{x_0}(x)}\right|d\mathfrak{m}(x)
\leq \mathscr{L}_{\mathfrak{m}}(x_0)\frac{2nA\varepsilon^{n-k}}{n-k}+A\int_{S_{x_0}M}d\nu_{x_0}(y)\int^{\min\{R,i_y\}}_\varepsilon \frac{\hat{\sigma}_{x_0}(r,y)}{r^k}dr<+\infty,
\end{align*}
which concludes the proof.
\end{proof}

\begin{proof}[Proof of Theorem \ref{Hardyineq}] Due  to Lemma \ref{implem}, the proof is similar to the one of Proposition \ref{CKNFirst}. Fix a point $x_0\in M$ and $u\in C^\infty_0(M)$; then we have
\begin{align*}
\ds\int_M \frac{u^2}{\rho_{x_0}^{ }}\Delta \rho_{x_0} d\mathfrak{m}=&-\ds\int_M\left\langle  d\left( \frac{u^2}{\rho^{ }_{x_0}}  \right), \nabla\rho_{x_0}  \right\rangle d\mathfrak{m}
=-2\ds\int_M \frac{u}{\rho_{x_0} }\left\langle du, \nabla\rho_{x_0}   \right\rangle d\mathfrak{m}+ \ds\int_M \frac{u^2}{\rho_{x_0}^2}d\mathfrak{m}\\
\leq& 2\left|\ds\int_M \frac{ u}{\rho_{x_0} }\left\langle du, \nabla\rho_{x_0}   \right\rangle d\mathfrak{m}   \right|+ \ds\int_M \frac{u^2}{\rho_{x_0}^2}d\mathfrak{m}. \tag{5.2}\label{5.6}
\end{align*}
As in  (\ref{3.0*}), set $M_-$, $M_+$ and $M_0$.
Now we have
\begin{align*}
\left|\ds\int_M \frac{ u}{\rho_{x_0} }\left\langle du, \nabla\rho_{x_0}   \right\rangle d\mathfrak{m}   \right|
\leq &\ds\int_{M_-} \frac{|u|}{\rho_{x_0}^{ }}\left\langle d(-u),\nabla\rho_{x_0}   \right\rangle d\mathfrak{m}+\ds\int_{M_+} \frac{|u|}{\rho_{x_0}^{ }}\left\langle du, \nabla\rho_{x_0}   \right\rangle d\mathfrak{m}\tag{5.3}\label{5.7}\\
\leq &\ds\int_{M_-} \frac{|u| }{\rho_{x_0}^{ }}F^*(-du)    d\mathfrak{m}+\ds\int_{M_+} \frac{|u| }{\rho_{x_0}^{ }}F^*(du) d\mathfrak{m}\tag{5.4}\label{5.8}\\
\leq& \ds\int_{M} \frac{|u| }{\rho_{x_0}^{ }}\max\left\{ F^*(\pm du)  \right\}  d\mathfrak{m}\\
\leq &\left(\ds\int_M \frac{u^{2} }{\rho_{x_0}^{2}} d\mathfrak{m}\right)^{\frac12}\left( \ds\int_M  \max\left\{ F^{*2}(\pm du)  \right\}   d\mathfrak{m}\right)^{\frac12},
\end{align*}
which together with (\ref{5.6}) yields
\begin{align*}
\left(\ds\int_M \frac{u^{2} }{\rho_{x_0}^{2}} d\mathfrak{m}\right) \left( \ds\int_M  \max\left\{ F^{*2}(\pm du)  \right\}    d\mathfrak{m}\right)\geq \frac{(n-2)^2}{4}\left(\ds\int_M \frac{u^2}{\rho_{x_0}^2}d\mathfrak{m} \right)^2.
\end{align*}

 Assume in the sequel that $\lambda_F(M)=1$ and let $(r,y)$ be the polar coordinate system around $x_0$.
First, we claim that the constant $(n-2)^2/4$ cannot be archived by  an extremal.
 Otherwise,  the equalities in (\ref{5.6})-(\ref{5.8}) furnish that the extremal must satisfy $u=u(r)$ together with $u\frac{\partial u}{\partial r}\leq 0$ and $\kappa   \frac{|u| }{r }=\left| \frac{\partial u}{\partial r}  \right|$ for some  $\kappa\geq0$.
Thus, $u=\frac{C}{r^\kappa}$ for some $C\in \mathbb{R}\setminus \{0\}$ and  $J_{2,2}(x_0,u)=J^{\rm max}_{2,2}(x_0,u)=\frac{(n-2)^2}{4}$ implies $\kappa=\frac{n-2}2$. However, in this case,  (\ref{new2.6}) together with (\ref{new2.5}) implies that for every $y\in S_{x_0}M$,
\[
\hat{\sigma}_{x_0}(r,y)\geq  e^{-{\tau}(y)} r^{n-1} \text{ for }0<r<{i}_y.\tag{5.5}\label{5.5new++}
\]
Hence, we have
\begin{align*}
\ds\int_M \frac{u^2(x)}{{\rho^{2}_{x_0}}(x)}d\mathfrak{m}=C^2\ds\int_{S_{x_0}M}d\nu_{x_0}(y)\int^{i_y}_0\frac{\hat{\sigma}_{x_0}(r,y)}{r^n}dr
\geq nC^2\mathscr{L}_{\mathfrak{m}}(x_0)\int^{\mathfrak{i}_{x_0}}_0\frac1rdr=+\infty,
\end{align*}
which proves that $\frac{(n-2)^2}4$ cannot be achieved by any function. In the sequel, we prove
\[
\inf_{u\in C^\infty_0(M)\backslash\{0\}}\frac{\ds\int_M  {F^{*2}(d u)} d\mathfrak{m}}{\ds\int_M \frac{u^2}{{\rho^{2}_{x_0}}}d\mathfrak{m}}=\frac{(n-2)^2}{4}=:\gamma^2.
\]

Given $0<\epsilon<r<R< {\mathfrak{i}}_{x_0}$, choose a cut-off function $\psi\in C^\infty_0(M)$ with $\text{supp}(\psi)=B_{x_0}(R)$ and $\psi|_{B_{x_0}(r)}\equiv1$.   Set $u_\epsilon(x):=\left[\max \{\epsilon, \rho_{x_0}(x)\}\right]^{-\gamma}$. Since $u:=\psi u_\epsilon\geq 0$ , we have
\begin{align*}
I_1(\epsilon):=&\ds\int_M  {F^{*2}(d u)} d\mathfrak{m}=\ds\int_{B_{x_0}(r)\backslash B_{x_0}(\epsilon)} {F^{*2}(\gamma \rho_{x_0}^{-\gamma-1} d  \rho_{x_0})} d\mathfrak{m} +\ds\int_{B_{x_0}(R)\backslash B_{x_0}(r)} {F^{*2}(d (\psi \rho^{-\gamma}_{x_0}))} d\mathfrak{m} \\
=&:\gamma^2 \mathfrak{I}_1+\mathfrak{I}_2,\tag{5.6}\label{2.6}
\end{align*}
where
\begin{align*}
\mathfrak{I}_1:=\ds\int_{B_{x_0}(r)\backslash B_{x_0}(\epsilon)}{\rho^{-n}_{x_0}} d\mathfrak{m},\ \mathfrak{I}_2:=\ds\int_{B_{x_0}(R)\backslash B_{x_0}(r)}{F^{*2}(d (\psi \rho^{-\gamma}_{x_0}))}d\mathfrak{m}.
\end{align*}
Clearly, $\mathfrak{I}_2$ is independent of $\epsilon$ and finite.
On the other hand, we have
\begin{align*}
I_2(\epsilon):=&\ds\int_M \frac{u^2(x)}{{\rho^{2}_{x_0}}(x)}d\mathfrak{m} (x)\geq \ds\int_{B_{x_0}(r)\backslash B_{x_0}(\epsilon)} \frac{(\psi u_\epsilon)^2(x)}{{\rho^{2}_{x_0}}(x)}d\mathfrak{m}(x)\\
=&\ds\int_{B_{x_0}(r)\backslash B_{x_0}(\epsilon)} \frac{\rho^{-2\gamma}_{x_0}(x)}{{\rho^{2}_{x_0}}(x)}d\mathfrak{m}(x)=\mathfrak{I}_1.\tag{5.7}\label{2.7}
\end{align*}

We now estimate $\mathfrak{I}_1$.  The co-area formula (\ref{new2.1}) then yields
\begin{align*}
\mathfrak{I}_1
=\int^r_\epsilon dt \ds\int_{ {S}_{x_0}(t)}t^{-n} d {A}=\int^r_\epsilon t^{-n}  {A}( {S}_{x_0}(t))\, dt,\tag{5.8}\label{2.8}
\end{align*}
where $S_{x_0}(t):=\{x\in M:\rho_{x_0}(x)=t\}$.
If $(t,y)$ is the polar coordinate system around $x_0$, (\ref{5.5new++}) yields
\begin{align*}
 {A}( {S}_{x_0}(t))=\ds\int_{ {S_{x_0}M}} \hat{\sigma}_{x_0}(t,y)d\nu_{x_0}(y)
 \geq \ds\int_{ {S_{x_0}M}}e^{- {\tau}(y)} t^{n-1} d\nu_{x_0}(y)=n\mathscr{L}_{\mathfrak{m}}(x_0) t^{n-1}.\tag{5.9}\label{2.9}
\end{align*}
Now (\ref{2.8}) combined with (\ref{2.9}) yields that
\[
\mathfrak{I}_1\geq n\mathscr{L}_{\mathfrak{m}}(x_0)\left[\ln r-\ln \epsilon \right]\rightarrow +\infty,\text{ as }\epsilon\rightarrow 0^+,
\]
which together with   (\ref{2.6}) and (\ref{2.7}) furnishes
\begin{align*}
\gamma^2\leq \inf_{u\in C^\infty_0(M)\backslash\{0\}}\frac{\ds\int_M   {F^{*2}(d u)} d\mathfrak{m}}{\ds\int_M \frac{u^2(x)}{{\rho^{2}_{x_0}}(x)}d\mathfrak{m}} \leq \lim_{\epsilon\rightarrow 0^+}\frac{I_1(\epsilon)}{I_2(\epsilon)}
=\lim_{\epsilon\rightarrow 0^+}\frac{\gamma^2\mathfrak{I}_1+\mathfrak{I}_2}{\mathfrak{I}_1}= \gamma^2,
\end{align*}
which concludes the proof.
\end{proof}

Similarly as in the proof of Theorem \ref{Hardyineq}, one can show the following result without reversibility.
\begin{theorem}\label{th-hardy} Given $n\geq 3$,
	let $(M,F, d\mathfrak{m})$ be an $n$-dimensional forward complete Finsler manifold with $\mathbf{K}\leq 0$ and  $\mathbf{S}\leq 0$. Then
	\begin{align*}
	\ds\int_M  \|du\|_{x,F}^2   d\mathfrak{m} \geq \frac{(n-2)^2}{4} \ds\int_M \frac{u^2}{\rho_{x}^2}d\mathfrak{m},\ \ \forall x\in M,\ u\in C^\infty_0(M).
	\end{align*}
	Moreover,  the constant $\frac{(n-2)^2}{4}$ is sharp but never achieved.
\end{theorem}

We conclude this section by formulating the following natural question.

\medskip

\noindent \textbf{Problem.} \textit{Under the same assumptions as in Theorem} \ref{th-hardy},  \textit{prove that for every $x_0\in M$},
\[
\inf_{u\in C^\infty_0(M)\backslash\{0\}}\frac{\ds\int_M  {F^{*2}(d u)} d\mathfrak{m}}{\ds\int_M \frac{u^2}{{\rho^{2}_{x_0}}}d\mathfrak{m}}=\frac{(n-2)^2}{4\lambda_F^2(M)}.\tag{5.9}\label{minek}
\]
Clearly, (\ref{minek}) trivially holds whenever $F$ is reversible,  see Theorem \ref{Hardyineq}. Moreover, in Farkas, Krist\'aly and Varga \cite{FKV} there is a non-reversible version of the Hardy inequality which also supports the above question. Finally, the Funk model $(M,F)=(B^n,F)$ -- mentioned in the Introduction and postponed to the Appendix -- also supports the above problem; indeed, in this case the reversibility is $\lambda_F(B^n)=+\infty$ thus the right hand side of (\ref{minek}) formally reduces to $0,$ as we already claimed in  (\ref{Funk-nulla}).

\section{Appendix}\label{Sec6}

\subsection{Examples from Introduction}\label{example} Although Theorems  \ref{fistCan}-\ref{forCan} provide a quite full picture on the validity of uncertainty principles and the existence of extremals on Finsler manifolds, in the sequel we present two examples  which provided the starting point of our study and show the optimality of our results. The first example emphasizes the role of the reversibility in uncertainty principles; the second example shows that in too general Finsler manifolds -- even with constant negative flag curvature (see Statement 1) --  the uncertainty principles may fail. Both examples are of Randers-type arising from the Zermelo navigation problem, see Bao, Robles and Shen \cite{BRS}.

\begin{example}  \label{simexa}\rm (cf. (\ref{elso-minko-norma}))
For a fixed $t\in [0,1)$, consider the space $(M,F_t)=(\mathbb{R}^2, F_t)$, where $$F_t(x,y):=\alpha+\beta=|y|+t y^2,\ \ y=(y^1,y^2)\in \mathbb R^2.$$
	Since there is no space-dependence in $F_t$, it turns out that $(\mathbb{R}^2, F_t)$ is Minkowskian with $\mathbf{K}=0$,  $\mathbf{S}_{BH}=0=\mathbf{S}_{HT}$ and $\mathfrak{i}(M)=+\infty.
$ In particular, $(\mathbb{R}^2, F_t)$ is a Berwaldian Randers-type Cartan-Haradamard manifold and its reversibility is $$\lambda_{F_t}(\mathbb R^2)=\frac{1+t}{1-t},$$ see  Farkas, Krist\'aly and Varga \cite[p. 1229]{FKV}.
	
	Let $(x^1,x^2)$ be the standard coordinate system of $\mathbb{R}^2$. Since $F_t$ is a Randers metric, we have
	\[
	d\mathfrak{m}_{BH}=(1-t^2)^{\frac{3}2}dx^1\wedge dx^2,\ d\mathfrak{m}_{HT}=dx^1\wedge dx^2.
	\]
	If $\textbf{0}=(0,0)$, since $F_t$ is a Minkowski metric (thus, it is translation-invariant),  a direct calculation yields
	\[
	\mathscr{L}_{\mathfrak{m}}\equiv \mathscr{L}_{\mathfrak{m}}(\textbf{0})=\frac12\ds\int_{S_{\textbf{0}}\mathbb R^2}e^{-\tau(y)}d\nu_{\textbf{0}}(y)=\left\{
	\begin{array}{lll}
	& \pi, & \ \ \ \text{for } \mathfrak{m}=\mathfrak{m}_{BH}, \\
	\\
	&\frac{\pi}{(1-t^2)^\frac32}, & \ \ \ \text{for } \mathfrak{m}=\mathfrak{m}_{HT}.
	\end{array}
	\right.
	\]
	On the other hand,  a geodesic in $(\mathbb{R}^2, F_t)$ is a straight line; therefore, one  gets
	\[
	\rho_{\textbf{0}}(x)=d_{F_t}(\textbf{0},x)=|x|+t x^2,\ \forall x=(x^1,x^2).
	\]
	According to Shen \cite[Example 3.2.1]{Sh1},
	we have
	\begin{align*}
	F_t^*(-d \rho_{\textbf{0}}(x))&=\frac{\sqrt{(1-t^2)|-d\rho_{\textbf{0}}|^2+t^2(-\partial_2 \rho_{\textbf{0}})^2}+t\partial_2 \rho_{\textbf{0}}}{1-t^2}\\
	&=\frac{1+t^2+2t\frac{x^2}{|x|}}{1-t^2}.\label{utolso-korul}\tag{6.1}
	\end{align*}
	
	It is easy to check that  $({ \mathbf{ J}}^{\rm max}_{2,0,\textbf{0}})$ holds. Now assume that  $n^2/4=1$ is achieved in $({ \mathbf{ J}}^{\rm max}_{2,0,\textbf{0}})$
	 by an extremal function $u$. Due to Proposition \ref{strongfirsthe-prop},  the extremal has the form $u:= e^{-C\rho_\textbf{0}^2}$ for $C>0$; for simplicity, set $C=1$. Note that
	\begin{align*}
	J^{\rm max}_{2,0}(\textbf{0},u)\geq \frac{\left(\ds\int_{\mathbb{R}^2} F_t^{*2}(du)d\mathfrak{m}\right)\left( \ds\int_{\mathbb{R}^2}  \rho^2_0 u^2 d\mathfrak{m} \right)}{\left(\ds\int_{\mathbb{R}^2}u^2d\mathfrak{m}\right)^2}=J_{2,0}(\textbf{0},u).
	\end{align*}
	An easy computation furnishes
	\begin{align*}
	&\ds\int_{\mathbb{R}^2} u^2 d\mathfrak{m}=2\mathscr{L}_{\mathfrak{m}}(\textbf{0})\int^\infty_0r e^{-2r^2}dr=\frac{\mathscr{L}_{\mathfrak{m}}(\textbf{0})}{2}, \\
	&\ds\int_{\mathbb{R}^2} \rho^2_{\textbf{0}} u^2 d\mathfrak{m}=2\mathscr{L}_{\mathfrak{m}}(\textbf{0}) \int^\infty_0r^3 e^{-2r^2}dr=\frac{\mathscr{L}_{\mathfrak{m}}(\textbf{0})}{4}.
	\end{align*}
	Similarly, by (\ref{utolso-korul}) we have that
	\begin{align*}
	&\ds\int_{\mathbb{R}^2} F_t^{*2}(du)d\mathfrak{m}=\mathscr{L}_{\mathfrak{m}}(\textbf{0})\frac{4-3\sqrt{1-t^2}}{\sqrt{1-t^2}}.
	\end{align*}
	Hence,
	\[
	J^{\rm max}_{2,0}(\textbf{0},u)\geq J_{2,0}(\textbf{0},u)=\frac{4-3\sqrt{1-t^2}}{\sqrt{1-t^2}} \geq 1,
		\]
	with equality if and only if $t=0$. Thus, $1=n^2/4$ is sharp in $({ \mathbf{ J}}^{\rm max}_{2,0,\textbf{0}})$    if and only if $t=0$, i.e., $F_t=F_0$ is reversible, in which case $d\mathfrak{m}_{BH}=d\mathfrak{m}_{HT}$ is precisely the Lebesgue measure on $\mathbb R^2$; this fact is in a perfect concordance with the statement of Theorem \ref{fistCan}/(ii). A similar argument shows (with the same candidate $u= e^{-\rho_\textbf{0}^2}$ for the extremal, cf. Proposition \ref{reverLemmaflag-prop}) that
	$$J_{2,0}(\textbf{0},u)=\frac{4-3\sqrt{1-t^2}}{\sqrt{1-t^2}} \geq \frac{1}{\lambda^2_{F_t}(\mathbb R^2)} =\frac{(1-t)^2}{(1+t)^2},$$
	with equality if and only if $t=0$, which confirms the statement of Theorem \ref{thirdCan}.
	
	One can also show by a direct computation that 	$
		J^{\rm min}_{2,0}(\textbf{0},u)\geq 1$ for every $u\in C_0^\infty(\mathbb R^2)\backslash\{0\}$  if and only if $ t=0$ i.e.,   $F_t=F_0$  is reversible; this fact supports Theorem \ref{forCan}.

\end{example}

\begin{example}\rm \label{example-2}  (cf. (\ref{Funk-nulla}))
Let $M:=\mathbb{B}^n=\{x\in \mathbb R^n:|x|<1\}$ be the $n$-dimensional Euclidean unit ball, $n\geq 3$, and consider the \textit{Funk metric} $F:\mathbb{B}^n\times
\mathbb R^{n}\to \mathbb R$  defined by
$$
F(x,y)=\frac{\sqrt{|y|^2-(|x|^2|y|^2-\langle
		x,y\rangle^2)}}{1-|x|^2}+\frac{\langle x,y\rangle}{1-|x|^2},\ x\in
\mathbb{B}^n,\ y\in T_x\mathbb{B}^n=\mathbb R^n.
$$
Hereafter, $|\cdot|$ and
$\langle\cdot, \cdot\rangle$ denote the $n$-dimensional Euclidean
norm and inner product. The pair $(\mathbb{B}^n,F)$ is a non-reversible Randers-type Finsler manifold,  see  Shen \cite{Sh1}, and its reversibility is $\lambda_F(\mathbb{B}^n)=+\infty,$ see Krist\'aly and Rudas \cite{KRudas}.  The dual Finsler metric of $F$ is
$$
F^*(x,y)=|y|-\langle
x,y\rangle
,\ \ \ (x,y)\in \mathbb{B}^n\times
\mathbb R^{n}.
$$
The distance function associated
to $F$ is
$$d_{F}(x_1,x_2)=\ln\frac{\sqrt{|x_1-x_2|^2-(|x_1|^2|x_2|^2-\langle x_1,x_2\rangle^2)}-\langle x_1,x_2-x_1\rangle}{\sqrt{|x_1-x_2|^2-(|x_1|^2|x_2|^2-\langle x_1,x_2\rangle^2)}-\langle x_2,x_2-x_1\rangle},\ x_1,x_2\in \mathbb{B}^n,$$
see Shen \cite[p.141 and p.4]{Sh1}; in particular,
\[
\rho_{\textbf{0}}(x)=d_{F}(\textbf{0},x)=-\ln(1-|x|) \,\ {\rm and}\ \ \varrho_{\textbf{0}}(x)=d_{F}(x,\textbf{0})=\ln(1+|x|),\ \ x\in \mathbb{B}^n,\label{utolso-korul-2}\tag{6.2}
 \]
 where $\textbf{0}=(0,...,0)\in \mathbb R^n$.
The Busemann-Hausdorff measure on $(\mathbb{B}^n,F)$ is $d\mathfrak{m}_{BH}(x)=dx,$ see Shen \cite[Example 2.2.4]{Sh1}. The Finsler manifold $(\mathbb{B}^n,F)$ is forward (but not backward) complete, it has constant negative flag curvature ${\bf K}= -\frac{1}{4}$, see Shen \cite[Example 9.2.1]{Sh1} and its $S$-curvature is ${\bf S}(x,y)=\frac{n+1}{2}F(x,y)$, $(x,y)\in T\mathbb R^n,$  see Shen \cite[Example 7.3.3]{Sh1}.

In the sequel we show that the Hardy inequality fails on $(\mathbb{B}^n,F)$; to do this, we recall by (\ref{1.3}) that
$$
J_{2,2}(\textbf{0},u)=\frac{\ds\int_{\mathbb{B}^n} F^{*2}(du)d\mathfrak{m}_{BH}}{\ds\int_{\mathbb{B}^n} \frac{u^2}{\rho^2_{\textbf{0}}}d\mathfrak{m}_{BH}},\ \ u\in C_0^\infty(\mathbb{B}^n)\setminus \{0\}.
$$
For every $\alpha>0$, let
$$
u_\alpha(x):=-e^{-\alpha \rho_{\textbf{0}}(x)}=-(1-|x|)^\alpha,\ \ x\in \mathbb{B}^n.
$$
Clearly, $u_\alpha$ can be approximated by functions belonging to $C_0^\infty(\mathbb{B}^n)$; moreover, $u_\alpha\in H_{0,F}^{1}(\mathbb{B}^n)$ for every $\alpha>0$, where  $H_{0,F}^{1}(\mathbb{B}^n)$ is the closure of $C_0^\infty(\mathbb{B}^n)$
with respect to the (positively homogeneous) norm
$$
\|u\|_{F}=\left(\displaystyle \int_{\mathbb{B}^n} F^{*2}(du)d\mathfrak{m}_{BH}+\displaystyle \int_{\mathbb{B}^n} u^2d\mathfrak{m}_{BH}\right)^{1/2}.
$$
Indeed, we have that  $F^*(du_\alpha(x))=\alpha(1-|x|)^\alpha,$ thus
$$\ds\int_{\mathbb{B}^n} F^{*2}(du_\alpha(x))d\mathfrak{m}_{BH}(x)=\alpha^2\ds\int_{\mathbb{B}^n} (1-|x|)^{2\alpha}dx=\alpha^2n\omega_n {\sf B}(2\alpha+1,n),$$
where $\omega_n$  and ${\sf B}$ denote the volume of the $n$-dimensional Euclidean unit ball and the Beta function, respectively. In a similar way, one has
$$\ds\int_{\mathbb{B}^n} u^2_\alpha(x)d\mathfrak{m}_{BH}(x)=n\omega_n {\sf B}(2\alpha+1,n).$$
Since $\ln^2(s)\leq {s^{-2}}$ for every $s\in (0,1]$, by (\ref{utolso-korul-2}) it turns out that
$$\ds\int_{\mathbb{B}^n} \frac{u_\alpha^2(x)}{\rho^2_{\textbf{0}}(x)}d\mathfrak{m}_{BH}(x)\geq \ds\int_{\mathbb{B}^n} (1-|x|)^{2\alpha+2}dx=n\omega_n {\sf B}(2\alpha+3,n).$$
Consequently,
$$\inf_{u\in C_0^\infty(\mathbb{B}^n)\setminus \{0\}}J_{2,2}(\textbf{0},u)\leq \inf_{\alpha>0}\frac{\ds\int_{\mathbb{B}^n} F^{*2}(du_\alpha)d\mathfrak{m}_{BH}}{\ds\int_{\mathbb{B}^n} \frac{u_\alpha^2}{\rho^2_{\textbf{0}}}d\mathfrak{m}_{BH}}\leq \inf_{\alpha>0}\alpha^2\frac{{\sf B}(2\alpha+1,n)}{{\sf B}(2\alpha+3,n)}=0,$$
which concludes the proof of (\ref{Funk-nulla}).
\end{example}

\subsection{Finsler manifolds with $\mathbf{K}=\mathbf{S}_{BH}=0$} In this subsection we discuss  more detailed the arguments from Remark \ref{elso-remark}/(ii).
We have seen throughout the paper that  Finsler manifolds verifying
\[
\mathbf{K}=0,\ \mathbf{S}_{BH}=0\tag{6.3}\label{3.10101010}
\]
play an important role in the study of uncertainty principles.
In the Riemmanian setting it is well-known  that such a manifold is locally isometric to the Euclidean space and particularly, it is globally isometric to the Euclidean space whenever it is complete and simply connected.

At this point, a natural question arises in the Finslerian setting: does a Finsler manifold verifying (\ref{3.10101010}) is locally isometric to a Minkowski space?

According to Berwald \cite{BW2} or Shen \cite[Proposition 8.2.4]{Shen2013}, a Finsler manifold is locally Minkowskian if and only if it is a flat Berwald manifold. Thus, a natural approach to answer the above question is to study if a  manifold satisfying (\ref{3.10101010}) is Berwaldian. It turns out that in general the answer is \textit{negative}. Indeed, Shen \cite{Shen2003} constructed the following example: if $n\geq 3$ and $\Omega=\{x=(x^1,x^2,\overline x)\in \mathbb R^2\times \mathbb R^{n-2}:(x^1)^2+(x^2)^2<1\}$ is a cylinder in $\mathbb R^n$ then the metric $\tilde F:T\Omega\to \mathbb R$ given by
\begin{equation}\label{fish-tank}
\tilde F(x,y)=\frac{\sqrt{(-x^2y^1+x^1y^2)^2+|y|^2(1-(x^1)^2-(x^2)^2)}-(-x^2y^1+x^1y^2)}{1-(x^1)^2-(x^2)^2},\ y=(y^1,y^2,\overline y)\in T_x\Omega,
\end{equation}
is a Finsler metric verifying (\ref{3.10101010}), but it is \textit{not} Berwaldian (thus, not Minkowskian).

In the sequel, we provide a method by means of which we can  construct a whole class of non-Berwald manifolds verifying (\ref{3.10101010}); such an argument is based on the navigation problem on manifolds.
To do this, let $V$ be a vector field on the Finsler manifold $(M,F)$ and suppose that
$F(V)<1$.  At each point $x\in M$,  by shifting the indicatrix $S_x M:= \{y\in T_xM: F(x,y)=1 \}$ along the vector
$-V_x$, we obtain a new indicatrix which corresponds to a new Minkowski norm $\tilde{F}_x$.  Equivalently, the norm
$\tilde{F}_x(y) = \tilde{F}(x,y)$ is the unique solution to the following nonlinear equation
\begin{equation*}\label{eq:navigation}
	F\left(x, \frac{y}{\tilde{F}(x,y)}+V_x\right) = 1.
\end{equation*}
In this way a new Finsler metric $\tilde{F}$ is obtained on $M$ which is produced by the {\it navigation data} $(F, V)$.
\begin{remark}\rm
Note that the navigation problem adopted here slightly
differs from those in Shen \cite{Shen2003} and Bao,  Robles and Shen \cite{BRS}, where $(F,-V)$  has been used instead of the navigation data $(F, V )$.
\end{remark}

 The following result relates $F$ and $\tilde F$ whenever the vector field $V$ is a Killing field of the metric $F$.

\begin{theorem}\label{Huang-thm} Assume that $V$ is a Killing field of the Finsler manifold $(M,F)$ with $F(V)<1$, and let $\tilde{F}$ be the Finsler metric produced by the navigation
	data $(F,V)$. Then we have the following:
	
\begin{itemize}
	\item[{\rm (a)}] The flag curvatures and $S$-curvatures   of $(M,F)$ and $(M, \tilde{F})$ are related by
	\[
	\widetilde{\mathbf{K}}(y, \cdot) = \mathbf{K}(\tilde{y}, \cdot) \ \ {and}\ \ {\widetilde{\mathbf{S}}_{BH}}(y) = {\mathbf{S}_{BH}}(\tilde{y}),
	\]
	where $\tilde{y} = y - F(x,y)V;$
	
	\item[{\rm (b)}] If $\psi_t$ is a one-parameter isometry group of the Finsler manifold $(M,F)$ which generates the Killing field $V$, then for each $F$-geodesic $\gamma:(a,b)\to M$, the
	curve $t\mapsto\psi_t\gamma(t)$ is a $\tilde{F}$-geodesic.
\end{itemize}	
\end{theorem}

\begin{proof}
	Property (b) and the first part of (a) are well-known by  Huang and Mo  \cite{MH2007, HM2011} and Foulon and Matveev \cite{FM}. In the sequel, we sketch the proof of the remaining part of (a) concerning the $S$-curvatures.
	Note that at every point $x\in M$ the indicatrices of $F$ and $\tilde{F}$ only differ by a translation $V_x$.
	Consequently, the Busemann-Hausdorff measures of these two metrics coincide, i.e., $\sigma_F(x) = \sigma_{\tilde{F}}(x)$. Now let $\xi$ and $\tilde{\xi}$ be the Reeb fields of these two metrics; they are vector fields on the co-sphere bundles
	which are the Legendre transformations of the sprays of $F$ and $\tilde{F}$, respectively.  It is proved in Huang and Mo \cite{MH2007, HM2011} that $\xi = \tilde{\xi} + X_V$,
	where $X_V$ is the complete lift of the vector field $V$ to the cotangent bundle.  When $V$ is a Killing field, it is easy to see that
	the one-parameter isometry group generated by $V$ will preserve the Busemann-Hausdorff measure, thus $X_V(\sigma_F) = 0$. Since  $\textbf{S} = \xi(\sigma_F)$,  we have
	$\tilde{\textbf{S}} =\tilde{\xi}(\sigma_{\tilde{F}}) = (\xi - X_V)(\sigma_F) = \xi (\sigma_F)=\textbf{S}$.  
\end{proof}

The above result implies that if $(M,F)$ is forward complete and the Killing field $V$ is also complete, then $(M,\tilde{F})$ is forward complete as well; moreover,  if $F$ satisfies (\ref{3.10101010}) then $\tilde{F}$ also satisfies (\ref{3.10101010}).  However, even if $F$ is a Berwald metric,  $\tilde{F}$ is not necessarily of Berwald type in general, as long as $V$ is \textit{not} a parallel vector field; this is the idea behind our construction. We conclude the paper with two examples falling into the latter class of metrics.

\begin{example}[Shen's fish tank]\rm
Let $F(x,y)=|y|$ be the standard Euclidean metric on $\mathbb{R}^n$ and $Q\in\mathbb{R}^{n\times n}$ be a skew-symmetric matrix. Then $V=V_x = Qx$ is a Killing field and the corresponding one-parameter isometry group is given by
$
	\psi_t(x) = \mathrm{e}^{tQ}x,\,t\in\mathbb{R}, \,x\in\mathbb{R}^n.
$
Now let $M$ be the region bounded by $F(-V)=|V|<1$.  Then the metric $\tilde{F}$ produced by the navigation
data $(F, V)$ on $M$ is of Randers type given by
\begin{equation}\label{tank-2}
\tilde{F}(x,y) = \frac{\sqrt{(1-|V|^2)|y|^2+\langle V,y\rangle^2}}{1-|V|^2}+\frac{\langle V,y\rangle}{1-|V|^2}.
\end{equation}
In particular, if $V(x)=(x^2,-x^1,\textbf{0})\in \mathbb{R}^n$, $n\geq 3$, then $M=\Omega$ is the interior of a cylinder  $(x^1)^2+(x^2)^2<1$ in $\mathbb R^n$ and $\tilde F$ is precisely the metric (\ref{fish-tank}) of Shen \cite{Shen2003};  this example is also referred as the Shen's fish tank. Note that $(M,\tilde F)$ it is not forward complete; indeed, geodesics of the form $t\mapsto \mathrm{e}^{tQ}(x+ty)$, when $y=x$,
  will eventually move out of $M$. We also note that (\ref{tank-2}) is precisely the Funk metric from Example \ref{example-2} whenever $V(x) = x$; with this choice,  $V(x)=x$ is a homothetic vector field but not a Killing one.


\end{example}

\begin{example}[Rigid motions of the plane]\rm
The rigid motions of the Euclidean plane can be written in matrix form and they constitute a Lie group
\[
	E(2) = \left\{\begin{bmatrix} A & b\\ 0 & 1\end{bmatrix}\; : \; A\in O(2), b\in\mathbb{R}^{2\times 1}\right\}.
\]
Its Lie algebra (the set of left invariant vector fields) has a basis
\[
	e_1 = \begin{bmatrix} 0 & 0 & 1 \\ 0 & 0 & 0 \\ 0 & 0 & 0 \end{bmatrix},\quad
	e_2 = \begin{bmatrix} 0 & 0 & 0 \\ 0 & 0 & 1 \\ 0 & 0 & 0 \end{bmatrix},\quad
	e_3 = \begin{bmatrix} 0 & -1 & 0 \\ 1 & 0 & 0 \\ 0 & 0 & 0 \end{bmatrix}.
\]
Let $\alpha$ be the Riemannian metric on $E(2)$ such that $\{e_1,e_2,e_3\}$ is an orthonormal basis at each point.
It is easy to check that $\alpha$ has vanishing sectional curvature and $e_3$ is a parallel vector field for $\alpha$.
Let $\beta$ be the dual $1$-form of $e_3$, then $$F=\alpha\phi(\beta/\alpha)$$ is a Berwald metric with vanishing flag
curvature for a suitably chosen function $\phi$. A typical example of this kind is the mountain slope metric of Matsumoto \cite{Matsumoto} describing the law of walking with a constant speed $v$ under the effect of gravity on a slope having the angle $\alpha\in [0,\pi/2)$ with respect to the horizontal plane; in this case,  $$\phi(s)=\left(v+\frac{g}{2}\sin(\alpha)s\right)^{-1}, \ s\geq 0,$$ where
$g\approx9.81$, assuming the structural condition  $g\sin\alpha < v$ is  fulfilled.

Now let $\hat{V}$ be the right-invariant vector field corresponding to $e_3$ and let $V:=\epsilon\hat{V}$, $\epsilon\in(0,1)$.
Then $V$ is a Killing field for $F$ and the inequality $F(V)<1$ holds in a neighborhood of the identity element. The Finsler metric $\tilde{F}$ produced by the navigation data $(F,V)$ on $M$ has vanishing flag curvature and vanishing $S$-curvature,
but it is not of Berwald type.  Note that $(M,\tilde{F})$ is also non-complete.
\end{example}

We conclude the paper with a remark concerning the non-completeness of the above metrics.

\begin{remark}\rm On one hand, according to Huang and Xue \cite{HX} and Shen \cite{BS},  if $(M,F)$ is a forward complete Finsler manifold with $\mathbf{K}\leq 0$,  then any bounded Killing field $V$ must be parallel.  
	The new metric $\tilde{F}$ is defined at points where $F(V)<1$, so it
is not defined on the whole manifold; this is the source of non-completeness. On the other hand, as far as we know, all the examples of either \textit{forward complete} or \textit{reversible} Finsler manifolds with (\ref{3.10101010})    are {always} \textit{Berwaldian} and hence, Minkowskian. It remains to fully characterize the Finsler manifolds with the aforementioned properties which will be considered elsewhere.
\end{remark}

\textbf{Acknowledgements}.
The authors thank the anonymous Referee for her/his valuable comments.

\end{document}